\documentclass[a4paper,11pt]{article}

\usepackage{stmaryrd}

\usepackage[T1]{fontenc}
\usepackage[utf8]{inputenc}
\usepackage{lipsum}
\usepackage{amsmath, amssymb, mathtools, amsthm}
\usepackage{comment}

\usepackage[usenames,dvipsnames]{xcolor}
\usepackage{graphicx}
\usepackage{minitoc}
\usepackage{tikz}
\usepackage{enumerate}
\usepackage{filecontents}
\usepackage[margin=0.96 in]{geometry}
\usepackage{bbm}

\usepackage[numbers,sort&compress]{natbib}
\usepackage[colorlinks=true]{hyperref}
\hypersetup{urlcolor=black, citecolor=red}
\usepackage{mathtools}
\usepackage{float}
\usepackage{xspace}
\usepackage{wrapfig}
\usepackage{comment}
\usepackage{hyperref}
\usepackage{datetime}

 \usepackage{epsfig}
 \usepackage{subcaption}
 \usepackage{multirow}
 \usepackage{lineno}
 \usepackage{fullpage}
 \usepackage[normalem]{ulem} 
 \usepackage{makeidx}
 
\usepackage[title]{appendix}
\usepackage{dsfont}

\newcommand{\curlyH}{\mathcal{H}}
\newcommand{\curlyI}{\mathcal{I}}
\newcommand{\curlyJ}{\mathcal{J}}
\newcommand{\curlyK}{\mathcal{K}}

\newcommand{\curlyP}{\mathcal{P}}

\newcommand{\curlyT}{\mathcal{T}}

\DeclarePairedDelimiter{\abs}{\lvert}{\rvert}
\DeclarePairedDelimiter{\norm}{\lVert}{\rVert}

\DeclarePairedDelimiter{\prt}{(}{)}

\newcommand \commentout[1] {}

\DeclareMathAlphabet{\mathup}{OT1}{\familydefault}{m}{n}
\newcommand{\dx}[1]{\mathop{}\!\mathup{d} #1}

\newcommand{\CORR}[1]{{\color{black}#1}}
\newcommand{\CORRdeux}[1]{{\color{black}#1}}

\theoremstyle{plain}
\newtheorem{thm}{Theorem}[section]
\newtheorem{lemma}[thm]{Lemma}

\theoremstyle{remark}
\newtheorem{remark}[thm]{\bf Remark}
\newtheorem{definition}[thm]{\bf Definition}

\newcommand{\ie}{\emph{i.e.}\;}
\newcommand{\cf}{\emph{cf.}\;}

\newcommand{\sign}{\mathrm{sign}}

\newcommand{\ddt}{\frac{\dx{}}{\dx{t}}}

\newcommand{\R}{\mathbb{R}}

\newcommand*{\dd}{\mathop{\kern0pt\mathrm{d}}\!{}}
\newcommand {\es}  {\varepsilon}
\makeatletter
\let\@fnsymbol\@arabic
\makeatother
\title{ \CORR{Effective interface conditions for a porous medium type problem}}
\author{Giorgia Ciavolella\thanks{Sorbonne Universit{\'e}, Inria, Universit\'{e} de Paris, Laboratoire Jacques-Louis Lions, UMR7598, 4 place Jussieu, 75005 Paris, France. Emails: giorgia.ciavolella@inria.fr, ndavid@math.univ-lyon1.fr, alexandre.poulain@univ-lille.fr} \thanks{Dipartimento di Matematica, Università di Roma "Tor Vergata", Italy} \thanks{Current affiliation: Team MONC, INRIA Bordeaux-Sud-Ouest, Institut de Mathématiques de Bordeaux, CNRS UMR 5251 \& Université de Bordeaux, 351 cours de la Libération, 33405 Talence Cedex, France}
\and
Noemi David\footnotemark[1] \thanks{Dipartimento di Matematica, Universit\'a di Bologna, Italy} \thanks{Current affiliation: Institute Camille Jordan, Université de Lyon 1, 69100 Villeurbanne, France}
\and
Alexandre Poulain\footnotemark[1] \thanks{Corresponding author} \thanks{Current affiliation: Laboratoire Paul Painlevé, UMR 8524 CNRS, Université de Lille, F-59000 Lille, France. } } 

\date{\today}

\begin{document}
\maketitle
\begin{abstract} 
Motivated by biological applications on tumour invasion through thin membranes, we study a porous-medium type equation where the density of the  cell population evolves under Darcy's law, assuming continuity of both the density and flux velocity on the thin membrane which separates two domains. The drastically different scales and mobility rates between the membrane and the adjacent tissues lead to consider the limit as the thickness of the membrane approaches zero. We are interested in recovering the \textit{effective interface problem}
and the transmission conditions on the limiting zero-thickness surface, formally derived by Chaplain \textit{et al.} (2019), which are compatible with nonlinear generalized Kedem-Katchalsky ones. Our analysis relies on \textit{a priori} estimates and compactness arguments as well as on the construction of a suitable extension operator which allows to deal with the degeneracy of the mobility rate in the membrane, as its thickness tends to zero.

\end{abstract}

\begin{flushleft}
    \noindent{\makebox[1in]\hrulefill}
\end{flushleft}
\begin{flushright}
    \noindent{\makebox[1in]\hrulefill}
\end{flushright}

\section{Introduction}
 
We consider a model of cell movement through a membrane where the \CORR{population density} ${u=u(t,x)}$ is driven by porous medium dynamics. We assume the domain to be an open and bounded set $\Omega\subset\R^3$. This domain $\Omega$ is divided into three open subdomains, $\Omega_{i,\es}$ for $i=1,2,3$, where $\es>0$ is the thickness of the intermediate membrane, $\Omega_{2,\es}$, see Figure~\ref{fig:domains}. In the three domains, the cells are moving with different constant mobilities, $\mu_{i,\es}$, for $i=1,2,3$, and they are allowed to cross the adjacent boundaries of these domains which are $\Gamma_{1,2,\es}$ (between $\Omega_{1,\es}$ and $\Omega_{2,\es}$) and $\Gamma_{2,3,\es}$ (between $\Omega_{2,\es}$ and $\Omega_{3,\es}$). 
\CORR{Then, we write $\Omega= \Omega_{1,\es}\cup \Omega_{2,\es} \cup \Omega_{3,\es}$, with $\Gamma_{1,2,\es}=\partial{\Omega}_{1,\es}\cap \partial{\Omega}_{2,\es}$, and $\Gamma_{2,3,\es}=\partial{\Omega}_{2,\es}\cap \partial{\Omega}_{3,\es}$.} The system reads as 

\begin{equation}\label{epspb}
	\left\{
	\begin{array}{rlll}
		&\partial_t u_{i,\es} - \mu_{i,\es} \nabla \cdot ( u_{i,\es} \nabla p_{i,\es}) = u_{i,\es} G(p_{i,\es})  & \text{ in } (0,T)\times\Omega_{i,\es}, & i=1,2,3,\\[1em] 
		&\mu_{i,\es} u_{i,\es} \nabla p_{i,\es}\cdot \boldsymbol{n}_{i,i+1} =  \mu_{i+1,\es} u_{i+1,\es} \nabla p_{i+1,\es}\cdot \boldsymbol{n}_{i,i+1}  &\text{ on } (0,T)\times\Gamma_{i,i+1,\es}, & i = 1,2,\\[1em]
		&u_{i,\es} = u_{i+1,\es}  &\text{ on } (0,T) \times\Gamma_{i,i+1,\es}, & i = 1,2,\\[1em]
		&u_{i,\es}=0 &\text{ on } (0,T) \times\partial\Omega.
	\end{array}
	\right.
\end{equation}
We denote by $p_{i,\es}$ the density-dependent pressure, which is given by the following power law
\begin{equation*}
    p_{i,\es} = u_{i,\es}^{\gamma}, \quad \text{ with } \; \gamma > 1.
\end{equation*}

In this paper, we are interested in studying the convergence of System~\eqref{epspb} as $\es\rightarrow 0$. When the thickness of the thin layer decreases to zero, the membrane collapses to a limiting interface, $\tilde{\Gamma}_{1,3}$, which separates two domains denoted by $\tilde\Omega_1$ and $\tilde\Omega_3$, see Figure~\ref{fig:domains}. \CORR{Then, the domain turns out to be $\Omega=\tilde\Omega_1 \cup \tilde{\Gamma}_{1,3}  \cup\tilde\Omega_3$.} 
We derive in a rigorous way the \textit{effective problem}~\eqref{effectivepb}, and in particular, the transmission conditions on the limit density, $\tilde{u}$, across the effective interface. Assuming that \CORR{the mobility} coefficients satisfy $\mu_{i,\es} >0 $ for $i=1,3$ 
and
\begin{equation*}
\lim_{\es\to 0}\mu_{1,\es}=\tilde\mu_{1}\in (0,+\infty), \quad \qquad \lim_{\es\to 0}\frac{\mu_{2,\es}}{\es}=\tilde\mu_{1,3} \in (0,+\infty), \quad \qquad\lim_{\es\to 0} \mu_{3,\es}=\tilde\mu_{3}\in (0,+\infty),
\end{equation*}
we prove that\CORR{, in a weak sense, solutions of Problem~\eqref{epspb} converge to solutions of the following system }
\begin{equation}\label{effectivepb}
	\left\{
	\begin{array}{rlll}
		&\partial_t \tilde{u}_{i} - \tilde\mu_{i} \nabla \cdot ( \tilde{u}_{i} \nabla \tilde{p}_{i}) = \tilde{u}_{i} G(\tilde{p}_{i})  & \text{ in } (0,T) \times\tilde\Omega_{i},&  i=1,3,\\[1em]
		&\tilde\mu_{1,3} \llbracket \Pi \rrbracket= \tilde\mu_{1} \tilde{u}_{1} \nabla\tilde{p}_1 \cdot \boldsymbol{\tilde{n}}_{1,3} =  \tilde\mu_{3} \tilde{u}_{3} \nabla\tilde{p}_3\cdot \boldsymbol{\tilde{n}}_{1,3}  &\text{ on } (0,T)\times\tilde\Gamma_{1,3},\\[1em]
		&\tilde{u}=0 &\text{ on } (0,T)\times\partial\Omega,
	\end{array}
	\right.
\end{equation}
where $\Pi$ satisfies $\Pi'(u)= u p'(u)$, namely
\begin{equation*}
	\Pi(u) := \frac{\gamma}{\gamma +1} u ^{\gamma+1}.
\end{equation*}
We use the symbol $\llbracket(\cdot) \rrbracket$ to denote the jump across the interface $\tilde\Gamma_{1,3}$, \ie
\begin{equation}
\llbracket \Pi \rrbracket:=\frac{\gamma}{\gamma+1}(\tilde{u}^{\gamma+1})_3- \frac{\gamma}{\gamma+1}(\tilde{u}^{\gamma+1})_{1},
\end{equation}
where the subscript indicates that $(\cdot)$ is evaluated as the limit to a point of the interface coming from the subdomain $\tilde\Omega_1$, $\tilde\Omega_3$, respectively. 
\begin{figure}[H]
    \centering
    \includegraphics[width=16cm]{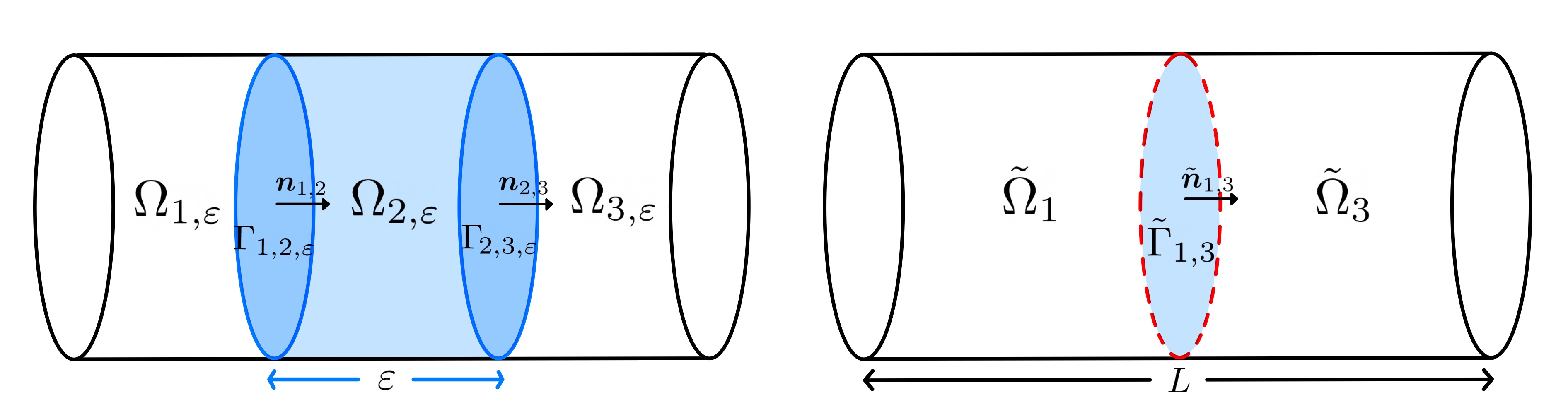}
    \caption{We represent here the bounded cylindrical domain $\Omega$ of length $L$. On the left, we can see the subdomains $\Omega_{i,\es}$ with related outward normals. The membrane $\Omega_{2,\es}$ of thickness $\es >0$ is delimited by $\Gamma_{i,i+1,\es}=\{x_3=\pm \es/2\}\cap\Omega$ which are symmetric with respect to the effective interface, $\tilde\Gamma_{1,3}=\{x_3=0\}\cap\Omega$. On the right, we represent the limit domain as $\es\rightarrow 0$. The effective interface, $\tilde\Gamma_{1,3}$, separates the two limit domains, $\tilde\Omega_1,\tilde\Omega_3$.}
    \label{fig:domains}
\end{figure}

\paragraph{Motivations and previous works.}

Nowadays, a huge literature can be found on the mathematical modeling of tumour growth {\color{black}, see, for instance, \cite{perthame,LFJCMWC, RCM, PT}}, on a domain $\Omega \subseteq \R^d$ (with $d=2,3$ for {\it in vitro} experiments, $d=3$ for {\it in vivo} tumours). Studying tumour's evolution, a crucial and challenging scenario is represented by cancer cells invasion through thin membranes. In particular, one of the most difficult barriers for the cells to cross is the \textit{basement membrane}. This kind of membrane separates the epithelial tissue from the connective one \CORR{(mainly consisting in extracellular matrix, ECM)}, providing a barrier that isolates malignant cells from the surrounding environment.
At the early stage, cancer cells proliferate locally in the epithelial tissue originating a carcinoma {\it in situ}. Unfortunately, cancer cells could mutate and acquire the ability to migrate by producing \textit{matrix metalloproteinases} (MMPs), specific enzymes which degrade the basement membrane\CORR{, allowing cancer cells to penetrate into it, invading the adjacent tissue.}    
A specific study can be done on the relation between MMP and their inhibitors as in Bresch 
\textit{et al.} \cite{bresch}. Instead, we are interested in modeling cancer transition from {\it in situ} stage to the invasive phase. \CORR{ This transition is described both by System~\eqref{epspb} and~\eqref{effectivepb}. In fact, for the both of them, the left domain can be interpreted as the domain in which the primary tumor lives, whereas the one on the right is the connective tissue. Between them, the basal membrane is penetrated by cancer cells either with a mobility coefficient (in the case of a nonzero thickness membrane) or with particular membrane conditions, in the case of a zero-thickness interface. }
 
Since in biological systems the membrane is often much smaller than the size of the other components, it is then convenient and reasonable to approximate the membrane as a zero-thickness one, as done in \cite{gallinato, giverso}, differently from \cite{bresch}. In particular, it is possible to mathematically describe cancer invasion through a zero thickness interface considering a limiting problem defined on two domains. The system is then closed by \textit{transmission conditions} on the effective interface which generalise the classical Kedem-Katchalsky conditions. The latter were first formulated in \cite{KK} and are used to describe different diffusive phenomena, such as, for instance, the transport of molecules through the cell/nucleus membrane \cite{cangiani, dimitrio, serafini}, solutes absorption processes through the arterial wall \cite{QVZ}, the transfer of chemicals through thin biological membranes \cite{calabro}, or the transfer of ions through the interface between two different materials \cite{bulicek}. \CORR{In our description, the transmission conditions define continuity of cells density flux through the effective interface $\tilde\Gamma_{1,3}$ and their proportionality to the jump of a term linked to cells pressure. The coefficient of proportionality is related to the permeability of the effective interface with respect to a specific population.}

For these reasons, studying  the convergence as the thickness of the membrane tends to zero represents a relevant and interesting problem both from a biological and mathematical point of view.
In the literature, this limit has been studied in different fields of applications other than tumour invasion, such as, for instance, thermal, electric or magnetic conductivity, \cite{sanchez-palencia, LRWZ}, or transport of drugs and ions through an heterogeneous layer, \cite{neuss-radu}. Physical, cellular and ecological applications characterised the bulk-surface model and the dynamical boundary value problem, derived in \cite{wang} in the context of boundary adsorption-desorption of diffusive substances between a bulk (body) and a surface. Another class of limiting systems is offered by \cite{LiWang}, in the case in which the diffusion in the thin membrane is not as small as its thickness. Again, this has a very large application field, from thermal barrier coatings (TBCs) for turbine engine blades to the spreading of animal species, from commercial pathways accelerating epidemics to cell membrane.

As it is now well-established, see for instance \cite{BD}, living tissues behave like compressible fluids. Therefore, in the last decades, mathematical models have been more and more focusing on the fluid mechanical aspects of tissue and tumour development, see for instance \cite{BC96, Green, BD, giverso, perthame, BCGR}. Tissue cells move through a porous embedding, such as the extra-cellular matrix (ECM). This nonlinear and degenerate diffusion process is well captured by filtration-type equations like the following, rather than the classical heat equation,
\begin{equation}\label{filtration}
		\partial_t u + \nabla \cdot (u {\bf v}) = F(u), \quad \text{ for }  t>0, \; x \in \Omega.
\end{equation}
Here $F(u)$ represents a generic density-dependent reaction term and the model is closed with the velocity field equation
\begin{equation}\label{eq: darcy}
    	{\bf v}:=-\mu \nabla p,
\end{equation}
and a density-dependent law of state for the pressure $p:=f(u).$
The function $\mu=\mu(t,x)\ge 0$ represents the cell mobility coefficient and the velocity field equation corresponds to the Darcy law of fluid mechanics. This relation between the velocity of the cells and the pressure gradient reflects the tendency of the cells to move away from regions of high compression. 

Our model is based on the one by Chaplain \textit{et al.} \cite{giverso}, where the authors formally recover the \textit{effective interface problem}, analogous to System~\eqref{effectivepb}, as the limit of a transmission problem, (or \textit{thin layer problem}) \cf System~\eqref{epspb}, when the thickness of the membrane converges to zero. They also validate through \CORR{ simulations the numerical equivalence between the two models.}
When shrinking the membrane $\Omega_{2,\es}$ to an infinitesimal region, $\tilde\Gamma_{1,3}$, (\ie when passing to the limit $\es\rightarrow 0$, where $\es$ is proportional to the thickness of the membrane), it is important to guarantee that the effect of the thin membrane on \CORR{cell invasion} remains preserved.
To this end, it is essential to make the following assumption on the mobility coefficient in the subdomain $\Omega_{2,\es}$,
\begin{equation*}
   \mu_{2,\es} \xrightarrow{\es\to 0}0 \qquad \text{ such that } \qquad \frac{\mu_{2,\es}}{\es}\xrightarrow{\es\to 0} \tilde\mu_{1,3}.
\end{equation*}
This condition implies that, when shrinking the pores of the membrane, the local permeability of the layer decreases to zero proportionally with respect to the local shrinkage. The function $\tilde\mu_{1,3}$ represents the \textit{effective permeability coefficient} of the limiting interface $\tilde\Gamma_{1,3}$, \ie the permeability of the zero-thickness membrane. We refer the reader to \cite[Remark 2.4]{giverso} for the derivation of the analogous assumption in the case of a fluid flowing through a porous medium. In \cite{giverso}, the authors derive the effective transmission conditions on the limiting interface, $\tilde\Gamma_{1,3}$, which relates the jump of the quantity $\Pi:= \Pi(u)$, defined by $\Pi'(u)=u f'(u)$ and the normal flux across the interface, namely
\begin{equation*}
    \tilde\mu_{1,3}  \llbracket \Pi \rrbracket= \tilde\mu_i \tilde{u}_i \nabla f( \tilde{u}_i)  \cdot \tilde{\boldsymbol{n}}_{1,3} =  \tilde\mu_i \nabla \Pi(\tilde{u}_i)  \cdot \tilde{\boldsymbol{n}}_{1,3}, \quad \text{ for } i=1,3 \quad\text{ on } \tilde\Gamma_{1,3}. \ \footnote{This equation is reported in \cite[Proposition 3.1]{giverso}, where we adapted the notation to that of our paper.}
\end{equation*}
These conditions turns out to be the well-known Kedem-Katchalsky interface conditions when $f(u):= \ln(u)$, for which $\Pi(u)= u + C$, $C\in \R$, \ie the linear diffusion case.

In this paper, we provide a rigorous proof to the derivation of these limiting transmission conditions, for a particular choice of the pressure law. To the best of our knowledge, this question has not been addressed before in the literature for a non-linear and degenerate model such as System~\eqref{epspb}. Although our system falls into the class of models formulated by Chaplain \textit{et al.}, we consider a less general case, making some choices on the quantities of interest. First of all, for the sake of simplicity, we assume the mobility coefficients $\mu_{i,\es}$ to be positive constants, hence they do not depend on time and space as in \cite{giverso}. We take a reaction term of the form $u G(p)$, where $G$ is a pressure-penalized growth rate. Moreover, we take a power-law as pressure law of state, \ie $p= u^\gamma$, with $\gamma\ge 1$. Hence, our model turns out to be in fact a porous medium type model, since Equations~(\ref{filtration}, \ref{eq: darcy}) read as follows
\begin{equation*}
    \partial_t u - \frac{\gamma}{\gamma+1}\Delta u^{\gamma+1} = u G(p), \quad \text{ for }  t>0, \; x \in \Omega.
\end{equation*}
The nonlinearity and the degeneracy of the porous medium equation (PME) bring several additional difficulties to its analysis compared to its linear and non-degenerate counterpart. In particular, the main challenge is represented by the emergence of a free boundary, which separates the region where $u>0$ from the region of vacuum. On this interface the equation degenerates, affecting the control and the regularity of the main quantities. For example, it is well-known that the density can develop jumps singularities, therefore preventing any control of the gradient in $L^2$, opposite to the case of linear diffusion. On the other hand, using the fundamental change of variables of the PME, $p=u^\gamma$, and studying the equation on the pressure rather than the equation on the density, turns out to be very useful when searching for better regularity of the gradient. Nevertheless, since the pressure presents "corners" at the free boundary, it is not possible to bound its laplacian in $L^2$ (uniformly on the entire domain). 

For these reasons, we could not straightforwardly apply some of the methods previously used in the literature in the case of linear diffusion. For instance, the result in \cite{BCF} is based on proving $H^2$-\textit{a priori} bounds, which do not hold in our case. The authors consider elliptic equations in a domain divided into three subdomains, each one contained into the interior of the other. The coefficients of the second-order terms are assumed to be piecewise continuous with jumps along the interior interfaces. Then, the authors study the limit as the thickness of the interior reinforcement tends to zero. 
In \cite{sanchez-palencia}, Sanchez-Palencia studies the same problem in the particular case of a lense-shaped region, $I_\es$, which shrinks to a smooth surface in the limit, facing also the parabolic case. The approach is based on $H^1$-\textit{a priori} estimates, namely the $L^2$-boundedness of the gradient of the unknown. Considering the variational formulation of the problem, the author is able to pass to the limit upon applying an extension operator. In fact, if the mobility coefficient in $I_\es$ converges to zero proportionally with respect to $\es$, it is only possible to establish uniform bounds outside of $I_\es$. The extension operator allows to "truncate" the solution and then "extend" it into $I_\es$ reflecting its profile from outside. Therefore, making use of the uniform control outside of the $\es$-thickness layer, the author is able to pass to the limit in the variational formulation. 
Let us also mention that, in the literature, one can find different methods and strategies for reaction-diffusion problems with a thin layer. For instance, in \CORR{\cite{maruvsic} the notion of two-scale convergence for thin domains is introduced which allows
the rigorous derivation of lower dimensional models. Some other papers have deepened the case of  heterogeneous membrane. We cite} \cite{neuss-radu}, \CORR{where} the authors develop a multiscale method which combines classical compactness results based on \textit{a priori estimates} and weak-strong two-scale convergence results in order to be able to pass to the limit in a thin heterogeneous membrane. \CORR{In \cite{gahn2022}, a transmission problem involving nonlinear diffusion in the thin layer is treated
and an effective model was derived. Finally, in \cite{gahn2021}, the accuracy of the effective approximations for processes through thin layers is studied by proving estimates for the difference between the original and the
effective quantities.
}
The passage at the limit allows to infer the existence of weak solutions for the effective Problem~\eqref{effectivepb}\CORRdeux{, thanks to the existence result for the $\es-$problem provided in Appendix \ref{appendix:existence}}. In the case of linear diffusion, the existence of global weak solutions for the effective problem with the Kedem-Katchalsky conditions is provided by \cite{ciavolella-perthame}. In particular, the authors prove it under weaker hypothesis such as $L^1$ initial data and reaction terms with sub-quadratic growth in an $L^1$-setting.

\paragraph{Outline of the paper.} 


The paper is organised as follows. In Section~\ref{hp}, we introduce the assumptions and notations, including the definition of weak solution of the original problem, System~\eqref{epspb}. In Section~\ref{sec: a priori}, \textit{a priori} estimates that will be useful to pass to the limit are proven.

Section~\ref{sec: limit} is devoted to prove the convergence of Problem~\eqref{epspb}, following the method introduced in \cite{sanchez-palencia} for the (non-degenerate) elliptic and parabolic cases.
The argument relies on recovering the $L^2$-boundedness (uniform with respect to $\es$) of the velocity field, in our case, the pressure gradient. As one may expect, since the permeability of the membrane, $\mu_{2,\es}$, tends to zero proportionally with respect to $\es$, it is only possible to establish a uniform bound outside of $\Omega_{2,\es}$. For this reason, following \cite{sanchez-palencia}, we introduce an extension operator (Subsection~\ref{subsec: extension}) and apply it to the pressure in order to extend the $H^1$-uniform bounds in the whole space $\Omega\setminus\tilde\Gamma_{1,3}$, hence proving compactness results.
We remark that the main difference between the strategy in \cite{sanchez-palencia} and our adaptation, is given by the fact that due to the non-linearity of the equation, we have to infer strong compactness of the pressure (and consequently of the density) in order to pass to the limit in the variational formulation. For this reason, we also need the $L^1$-boundedness of the time derivative, hence obtaining compactness with a standard Sobolev's embedding argument.
Moreover, since \commentout{the solution}\CORR{solutions} to the limit Problem~\eqref{effectivepb} will present discontinuities at the effective interface, we need to build proper test functions which belong to $H^1(\Omega\setminus \tilde \Gamma_{1,3})$ that are zero on $\partial \Omega$ and are discontinuous across $\tilde\Gamma_{1,3}$, (Subsection~\ref{subsec: testfnc}).

Finally, using the compactness obtained thanks to the extension operator, we are able to prove the convergence of \commentout{a solution}\CORR{solutions} to Problem~\eqref{epspb} to \commentout{a couple}\CORR{couples}  $(\tilde u, \tilde p)$ which \commentout{satisfies}\CORR{satisfy} Problem~\eqref{effectivepb} in a weak sense, therefore inferring the existence of \commentout{a solution}\CORR{solutions} of the effective problem, as stated in the following theorem. 

\begin{thm}[Convergence to the effective problem]\label{thm: main}
\CORR{Solutions of Problem~\eqref{epspb} converge weakly to solutions $(\tilde{u},\tilde{p})$ of Problem~\eqref{effectivepb} in the following weak form}
\begin{equation}\label{eq: variational limit problem}
\begin{split}
        -\int_0^T\int_{\Omega} \tilde{u} & \partial_t w +\tilde\mu_1 \int_0^T \int_{\tilde\Omega_1}   \tilde{u}  \nabla \tilde{p} \cdot \nabla w + \tilde\mu_3 \int_0^T\int_{\tilde\Omega_3}   \tilde{u}  \nabla \tilde{p} \cdot \nabla w \\
        &+ \tilde\mu_{1,3}\int_0^T\int_{\tilde\Gamma_{1,3}} \llbracket \Pi \rrbracket \prt*{w_{|x_3=0^+}- w_{|x_3=0^-}} = \int_0^T\int_{\Omega} \tilde{u}G(\tilde{p}) w + \int_{\Omega} \tilde{u}^0 w^0,
        \end{split}
\end{equation}
for all test functions $w(t,x)$ with a proper regularity (defined in Theorem~\ref{thm: existence}) and $w(T,x) = 0$ a.e. in $\Omega$. We used the notation
\[
\llbracket \Pi \rrbracket:=\frac{\gamma}{\gamma+1}(\tilde{u}^{\gamma+1})_{|x_3=0^+}- \frac{\gamma}{\gamma+1}(\tilde{u}^{\gamma+1})_{|x_3=0^-},
\]
\CORR{and $(\cdot)_{|x_3=0^-}= \curlyT_1 (\cdot)$ as well as $(\cdot)_{|x_3=0^+}= \curlyT_3 (\cdot)$, with $\curlyT_1, \curlyT_3$ the trace operators defined in Section~\ref{hp}.}
\end{thm}
\noindent Section~\ref{sec: conclusion} concludes the paper and provides some research perspectives.


\section{Assumptions and notations} \label{hp} 
Here, we detail the problem setting and assumptions.
For the sake of simplicity, we consider as domain $\Omega\subset \R^3$ a cylinder with axis $x_3$, see Figure~\ref{fig:domains}. Let us notice that it is possible to take a more general domain $\hat \Omega$ defining a proper diffeomorfism $F: \hat \Omega \to \Omega $. Therefore, the results of this work extend to more general domains as long as the existence of the map $F$ can be proved (this implies that $\hat \Omega$ is a connected open subset of $\R^d$ and has a smooth boundary). Therefore, we assume that the domain $\Omega$ has a $C^1$-piecewise boundary. We also want to emphasize the fact that our proofs hold in a 2D domain considering three rectangular subdomains.  
We introduce
\begin{equation*}
u_{\es} :=\left\{
\begin{array}{ll}
u_{1,\es}, \quad \mbox{ in } \; \Omega_{1,\es},\\[0.1em]
u_{2,\es}, \quad \mbox{ in } \; \Omega_{2,\es},\\[0.1em]
u_{3,\es}, \quad \mbox{ in } \; \Omega_{3,\es},
\end{array}
\right.
\qquad
  p_{\es} := \left\{
\begin{array}{ll}
p_{1,\es}, \quad \mbox{ in } \; \Omega_{1,\es},\\[0.1em]
p_{2,\es}, \quad \mbox{ in } \; \Omega_{2,\es},\\[0.1em]
p_{3,\es}, \quad \mbox{ in } \; \Omega_{3,\es}.
\end{array}
\right.
\end{equation*}

We define the interfaces between the domains $\Omega_{i,\es}$ and $\Omega_{i+1,\es}$ for $i=1,2$, as
\[
    \Gamma_{i,i+1,\es} = \partial \Omega_{i,\es} \cap\, \partial \Omega_{i+1,\es}.
\]
We denote with $\boldsymbol{n}_{i,i+1}$ the outward normal to $\Gamma_{i,i+1,\es}$ with respect to $\Omega_{i,\es}$, for $i=1,2$. Let us notice that $\boldsymbol{n}_{i,i+1}=- \boldsymbol{n}_{i+1,i}$. 

\CORR{
We define two trace operators
\[
\begin{cases}
\curlyT_1:  W^{k,p}(\tilde \Omega_1) \longrightarrow L^{p}(\partial \tilde \Omega_1),\\
\curlyT_3:  W^{k,p}(\tilde \Omega_3) \longrightarrow L^{p}(\partial \tilde \Omega_3),
\end{cases} 
    \quad \mbox{ for } 1\leq p < +\infty, \quad k\geq 1 .
\]
Therefore, for any $z  \in W^{k,p}(\Omega \setminus \tilde\Gamma_{1,3})$, we have the following decomposition 
\[
     z := \begin{cases}
     z_1,\quad \text{in} \quad \tilde \Omega_1,\\
     z_3,\quad \text{in} \quad \tilde \Omega_3.
     \end{cases}
\]
Obviously, we have that $z_\alpha \in  W^{k,p}(\tilde \Omega_\alpha)$ ($\alpha =1,3$). Thus, we denote 
\[
z_{|_{\partial \tilde \Omega_\alpha} }:=\curlyT_\alpha z \in L^{p}(\partial \tilde \Omega_\alpha ),\quad \alpha=1,3,
\]
and the following continuity property holds~\cite{brezis}
\begin{equation*}
    \|\curlyT_\alpha z\|_{L^{p}(\partial \tilde \Omega_\alpha )} \leq C \|z\|_{W^{k,p}(\tilde \Omega_\alpha) },\quad \alpha=1,3.
\end{equation*}
We assume $W^{k,p}(\Omega \setminus \tilde\Gamma_{1,3})$ is endowed with the norm
\CORRdeux{
\[ \| z\|_{W^{k,p}(\Omega \setminus \tilde\Gamma_{1,3})}= \|z\|_{L^p(\Omega \setminus \tilde\Gamma_{1,3})}+ \sum_{j=1}^k \| D^j z\|_{L^p(\Omega \setminus \tilde\Gamma_{1,3})}. \]
}}

We make the following assumptions on the initial data: there exists a positive constant $p_H$, such that
\begin{equation}\label{eq: assumptions initial data} \tag{A-data1}
   0\le p^0_\es \le p_H, \qquad 0\le u^0_\es \le p_H^{1/\gamma} =: u_H,  
\end{equation}
\begin{equation}\label{eq: assumptions initial data 2} \tag{A-data2}
   \Delta \prt*{(u^0_{i,\es})^{\gamma+1}} \in L^1(\Omega_{i,\es}), \quad \text{ for } i=1,2,3.
\end{equation} 
Moreover, we assume that there exists a function $\tilde u_0\in L^1_+(\Omega)$ \CORR{(\ie $ \tilde u_0\in L^1(\Omega)$ and non-negative)} such that
\begin{equation}\label{conv u0} \tag{A-data3}
    \|u_{\es}^0 - \tilde u_{0}\|_{L^1(\Omega)}\longrightarrow 0, \quad\text{ as } \es\rightarrow 0.
\end{equation}
The growth rate $G(\cdot)$ satisfies
\begin{equation}\label{eq: assumption G} \tag{A-G}
    G(0)=G_M>0, \quad G'(\cdot)<0, \quad G(p_H)=0.
\end{equation}
The value $p_H$, called \textit{homeostatic pressure}, represents the lowest level of pressure that prevents cell multiplication due to contact-inhibition.

We assume that the \CORR{mobility} coefficients satisfy $\mu_{i,\es} >0 $ for $i=1,3$ 
and
\begin{equation}
    \lim_{\es\to 0}\mu_{1,\es}=\tilde\mu_{1}>0, \qquad \qquad \lim_{\es\to 0}\frac{\mu_{2,\es}}{\es}=\tilde\mu_{1,3}>0, \qquad \qquad\lim_{\es\to 0} \mu_{3,\es}=\tilde\mu_{3}>0.
    \label{eq:cond-mu}
\end{equation}

\noindent{\textbf{Notations.}} For all $T>0$, we denote $\Omega_T:=(0,T)\times\Omega$. We use the abbreviated form ${u_\es:=u_\es(t):=u_\es(t,x).}$ From now on, we use $C$ to indicate a generic positive constant independent of $\es$ that may
change from line to line. Moreover, we denote 
	\begin{equation*}
	\sign_+{(w)}=\mathds{1}_{\{w>0\}},\quad  \quad \sign_-{(w)}=-\mathds{1}_{\{w<0\}},
	\end{equation*}
	and
	\[
	    \sign(w) = 	\sign_+{(w)} + 	\sign_-{(w)}.
	\]
	We also define the positive and negative part of $w$ as follows
	\begin{equation*}
	{\color{black}(w)_+} :=
	\begin{cases}
	w, &\text{ for } w>0,\\
	0, &\text{ for } w\leq 0,
	\end{cases}
	\quad  \text{ and }\quad
	{\color{black}(w)_-} :=
	\begin{cases}
	-w, &\text{ for } w<0,\\
	0, &\text{ for } w\geq 0.
	\end{cases}
	\end{equation*} 
{\color{black} We denote $|w|:=(w)_++(w)_-$.}
\bigskip

\noindent
Now, let us write the variational formulation of Problem~\eqref{epspb}.
\begin{definition}[Definition of weak solutions]
Given $\es>0$, a weak solution to Problem~\eqref{epspb} is given by $u_\es, p_\es \in L\CORR{^\infty}(0,T;L^\infty(\Omega))$ such that $\nabla p_\es \in L^2(\Omega_T)$ and
\begin{equation}
      -  \int_0^T\int_{\Omega} u_\es \partial_t \psi   + \sum_{i=1}^3  \mu_{i,\es} \int_0^T\int_{\Omega_{i,\es}} u_{i,\es} \nabla p_{i,\es} \cdot \nabla \psi  =  \int_0^T\int_{\Omega} u_{\es} G(p_{\es}) \psi  + \int_{\Omega} u^0_\es \psi (0,x) ,
        \label{eq:variational-pb-homo}
\end{equation} 
for all test functions $\psi \in H^1(0,T; H^1_0(\Omega))$ such that $\psi(T,x)=0$ a.e. in $\Omega$.
\end{definition} 


\section{A priori estimates}\label{sec: a priori}
We show that the main quantities satisfy some uniform \textit{a priori} estimates which will later allow us to prove strong  compactness and pass to the limit. 
\begin{lemma}[A priori estimates]\label{lemma: a priori}
Given the assumptions in Section~\ref{hp}, let $(u_\es, p_\es)$ be \CORR{a solution} of Problem~\eqref{epspb}. 
{\color{black} There exists a positive constant $C$ independent of $\es$ such that
\begin{itemize}
    \item[\textit{(i)}] 
    $0\le u_{\es}\le u_H$ and   $0\le p_{\es}\le p_H$,  
    \commentout{\item[\textit{(ii)}] 
    $\|u_{\es}\|_{L^\infty(0,T; L^p(\Omega))}\leq C, \; \|p_{\es}\|_{L^\infty(0,T; L^p(\Omega))} \leq C$, for $1\le p < \infty$,}
    \item[\textit{(ii)}] 
    $\|\partial_t u_\es\|_{L^\infty(0,T; L^1(\Omega))}\leq C,\; \|\partial_t p_\es\|_{L^\infty(0,T; L^1(\Omega))}\leq C$,
    \item[\textit{(iii)}] 
    $\|\nabla p_\es\|_{L^2(0,T; L^2(\Omega\setminus\Omega_{2,\es}))}\leq C$.
\end{itemize}}
\end{lemma}
\CORR{\begin{remark}
We remark that statement (i) implies that for all $p\in [1,\infty]$, we have \[\|u_{\es}\|_{L^\infty(0,T; L^p(\Omega))}\leq C, \; \|p_{\es}\|_{L^\infty(0,T; L^p(\Omega))} \leq C.\]
\end{remark}}
\CORRdeux{\begin{remark}
The following proof can be made rigorous by performing a parabolic regularization of the problem, namely by adding $\delta \Delta u_{i,\es}$, for $\delta>0$, to the left-hand side of the equation and in the flux continuity conditions. In fact, the following estimates can be obtained uniformly both in $\es$ and $\delta$.
\end{remark}}
\begin{proof}
Let us recall the equation satisfied by $u_\es$ on $\Omega_{i,\es}$, namely
\begin{equation}\label{eq: u}
    \partial_t u_{i,\es} - \mu_{i,\es} \nabla \cdot ( u_{i,\es} \nabla u_{i,\es}^{\gamma} )= u_{i,\es} G(p_{i,\es}).
\end{equation} 
\noindent{\bf(i)} $\boldsymbol{ 0\le u_{\es}\le u_H,\quad 0\le p_{\es}\le p_H.}$ \

The $L^\infty$-bounds of the density and the pressure are a straight-forward consequence of the comparison principle applied to Equation~\eqref{eq: u}, which can be rewritten as
\begin{equation}\label{eq: u2}
    \partial_t u_{i,\es} - \frac{\gamma}{\gamma+1} \mu_{i,\es} \Delta u_{i,\es}^{\gamma+1}= u_{i,\es} G(p_{i,\es}).
\end{equation} 
Indeed, summing up Equations~\eqref{eq: u2} for $i=1,2,3$, we obtain
\begin{equation}\label{eq: usum}
    \sum_{i=1}^3 \partial_t  u_{i,\es} - \frac{\gamma}{\gamma+1} \sum_{i=1}^3 \mu_{i,\es} \Delta u_{i,\es}^{\gamma+1}= \sum_{i=1}^3 u_{i,\es} G(p_{i,\es}).
\end{equation}
Then, we also have 
\begin{equation*}
    \sum_{i=1}^3 \partial_t (u_H-u_{i,\es})= \frac{\gamma}{\gamma+1} \sum_{i=1}^3 \mu_{i,\es} \Delta (u_H^{\gamma+1}-u_{i,\es}^{\gamma+1})+\sum_{i=1}^3 (u_H-u_{i,\es}) G(p_{i,\es}) - u_H\sum_{i=1}^3 G(p_{i,\es}).
\end{equation*}
\noindent
{\color{black}Let us recall Kato's inequality, \cite{kato}, \ie
$$\Delta (u)_-\geq \sign_-(u) \Delta u.$$}
If we multiply by $\sign_{-}(u_H-u_{i,\es})$, thanks to Kato's inequality, we infer that
\begin{equation}\label{eq: sign -}
\begin{split}
     \sum_{i=1}^3 \partial_t (u_H-u_{i,\es})_{-}&\leq \sum_{i=1}^3 \left[ \frac{\gamma}{\gamma+1} \mu_{i,\es} \Delta (u_H^{\gamma+1}-u_{i,\es}^{\gamma+1})_{-}+ (u_H-u_{i,\es})_{-} G(p_{i,\es}) \right.\\[0.3em]
     &\qquad \qquad - u_H   G(p_{i,\es})\sign_{-}(u_H-u_{i,\es})\Big]\\[0.3em]
     &\leq \sum_{i=1}^3\left[\frac{\gamma}{\gamma+1} \mu_{i,\es} \Delta (u_H^{\gamma+1}-u_{i,\es}^{\gamma+1})_{-}+ (u_H-u_{i,\es})_{-} G(p_{i,\es})\right],
     \end{split}
\end{equation}
where we have used the assumption \eqref{eq: assumption G}.
We integrate over the domain $\Omega$. Thanks to the boundary conditions in System~\eqref{epspb}, \ie the density and flux continuity across the interfaces, and the homogeneous Dirichlet conditions on $\partial \Omega$, we gain 
\begin{align*}
  \sum_{i=1}^3&\int_{\Omega_{i,\es}} \mu_{i,\es} \Delta (u_H^{\gamma+1}-u_{i,\es}^{\gamma+1})_{-} 
  \\[0.3em]
  &=\sum_{i=1}^2 \int_{\Gamma_{i,i+1,\es}} \left[\mu_i \nabla (u_H^{\gamma+1}-u_{i,\es}^{\gamma+1})_- - \mu_{i+1,\es}  \nabla (u_H^{\gamma+1}-u_{i+1,\es}^{\gamma+1})_-\right]\cdot \boldsymbol{n}_{i,i+1} \\[0.3em]
  &=\sum_{i=1}^2 \left[\int_{\Gamma_{i,i+1,\es}\cap \{u_H<u_{i,\es}\}} \mu_i \nabla u_{i,\es}^{\gamma+1}\cdot \boldsymbol{n}_{i,i+1}  - \int_{\Gamma_{i,i+1,\es}\cap \{u_H<u_{i+1,\es}\}} \mu_{i+1,\es}  \nabla u_{i+1,\es}^{\gamma+1} \cdot \boldsymbol{n}_{i,i+1}\right] \\[0.3em]
  &= \sum_{i=1}^2 \int_{\Gamma_{i,i+1,\es}\cap \{u_H<u_{i,\es}\}} \left[ \mu_i \nabla u_{i,\es}^{\gamma+1}- \mu_{i+1,\es}  \nabla u_{i+1,\es}^{\gamma+1} \right]\cdot \boldsymbol{n}_{i,i+1} \\[0.3em]
  &=0.
\end{align*}
Hence, from Equation~\eqref{eq: sign -}, we find
\begin{equation*}
    \frac{d}{dt}\sum_{i=1}^3 \int_{\Omega_{i,\es}}(u_H-u_{i,\es})_{-} \leq G_M \sum_{i=1}^3 \int_{\Omega_{i,\es}}(u_H-u_{i,\es})_{-}.
\end{equation*}
Finally, Gronwall's lemma and hypothesis~\eqref{eq: assumptions initial data} on $u_{i,\es}^0$ imply 
\begin{equation*}
   \sum_{i=1}^3 \int_{\Omega_{i,\es}}(u_H-u_{i,\es})_{-} \leq e^{G_M t} \sum_{i=1}^3\int_{\Omega_{i,\es}}(u_H-u^0_{i,\es})_{-} =0.
\end{equation*}
We then conclude the boundedness of $u_{i,\es}$ by $u_H$ for all $i=1,2,3$. From the relation $p_{\es}=u^\gamma_{\es}$, we conclude the boundedness of $p_\es$.

{\color{black} By arguing in an analogous way, replacing $u_H$ by $0$ and multiplying by $\sign_+(u_{i,\es})$, we obtain
$$ \sum_{i=1}^3 \int_{\Omega_{i,\es}}(u_{i,\es})_{-} \leq e^{G_M t} \sum_{i=1}^3\int_{\Omega_{i,\es}}(u^0_{i,\es})_{-} =0,$$
namely, $u_\es\geq 0$, and consequently, $p_\es\geq 0.$}

\CORR{\noindent{\bf(ii)}} $\boldsymbol{\partial_t u_\es, \partial_t p_\es \in L^\infty(0,T; L^1(\Omega)).}$ \

We derive Equation~\eqref{eq: u2} with respect to time to obtain
\begin{equation*}
    \partial_t (\partial_t u_{i,\es}) = \mu_{i,\es} \gamma \Delta\prt*{p_{i,\es} \partial_t u_{i,\es}} + \partial_t u_{i,\es} G(p_{i,\es}) + u_{i,\es} G'(p_{i,\es}) \partial_t p_{i,\es}.
\end{equation*}
Upon multiplying by $\sign(\partial_t u_{i,\es})$ and using Kato's inequality, we have
\begin{equation*}
    \partial_t (|\partial_t u_{i,\es}|) \le \mu_{i,\es} \gamma \Delta\prt*{p_{i,\es} |\partial_t u_{i,\es}|} + |\partial_t u_{i,\es}| G(p_{i,\es}) + u_{i,\es} G'(p_{i,\es}) |\partial_t p_{i,\es}|,
\end{equation*}
since $u_{i,\es}$ and $p_{i,\es}$ are both nonnegative and $\partial_t p_{i,\es}= \gamma u_{i,\es}^{\gamma-1}\partial_t u_{i,\es}$.
We integrate over $\Omega_{i,\es}$ and we sum over $i=1,2,3$, namely
\begin{equation}\label{eq: dt dtu}
    \ddt \sum_{i=1}^3  \int_{\Omega_{i,\es}} |\partial_t u_{i,\es}| \le \gamma \underbrace{\sum_{i=1}^3 \mu_{i,\es} \int_{\Omega_{i,\es}} \Delta\prt*{p_{i,\es} |\partial_t u_{i,\es}|}}_{\curlyJ} + G_M \int_{\Omega_{i,\es}} |\partial_t u_{i,\es}|,
\end{equation}
where we use that $G'\leq 0$.

Now we show that the term $\curlyJ$ vanishes.
Integration by parts yields
\begin{equation*}
    \curlyJ = \sum_{i=1}^2 \int_{\Gamma_{i,i+1,\es}} \mu_{i,\es} \nabla(p_{i,\es} |\partial_t u_{i,\es}|)\cdot \boldsymbol{n}_{i,i+1} + \sum_{i=1}^2 \int_{\Gamma_{i,i+1,\es}} \mu_{i+1,\es}\nabla(p_{i+1,\es} |\partial_t u_{i+1,\es}|)\cdot\boldsymbol{n}_{i+1,i}.
\end{equation*}
For the sake of simplicity, we denote $\boldsymbol{n} := \boldsymbol{n}_{i,i+1}.$ Let us recall that, by definition, $\boldsymbol{n}_{i+1,i}= -\boldsymbol{n}$.
We have
\begin{align*}
    \curlyJ = &\sum_{i=1}^2 \int_{\Gamma_{i,i+1,\es}} \prt*{\mu_{i,\es} \nabla(p_{i,\es} |\partial_t u_{i,\es}|) - \mu_{i+1,\es} \nabla(p_{i+1,\es} |\partial_t u_{i+1,\es}|)}\cdot\boldsymbol{n}\\
    = &\underbrace{\sum_{i=1}^2 \int_{\Gamma_{i,i+1,\es}} |\partial_t u_{i,\es}| \mu_{i,\es} \nabla p_{i,\es} \cdot \boldsymbol{n} - |\partial_t u_{i+1,\es}| \mu_{i+1,\es} \nabla p_{i+1,\es}\cdot\boldsymbol{n}}_{\curlyJ_1}\\
    &\; +\underbrace{\sum_{i=1}^2 \int_{\Gamma_{i,i+1,\es}} \mu_{i,\es} p_{i,\es}\nabla |\partial_t u_{i,\es}| \cdot \boldsymbol{n} - \mu_{i+1,\es} p_{i+1,\es} \nabla|\partial_t u_{i+1,\es}| \cdot \boldsymbol{n} }_{\curlyJ_2}.
\end{align*}
Let us recall the membrane conditions of Problem~\eqref{epspb}, namely
\begin{align}
		\mu_{i,\es} u_{i,\es} \nabla p_{i,\es}\cdot \boldsymbol{n} &=  \mu_{i+1,\es} u_{i+1,\es} \nabla p_{i+1,\es}\cdot \boldsymbol{n},  \label{cond: 1} \\[0.2em]
		u_{i,\es} &= u_{i+1,\es}, \label{cond: 2}
	\end{align}
on $ (0,T) \times\Gamma_{i,i+1,\es}$, for  $i = 1,2$.
From Equation~\eqref{cond: 2}, it is immediate to infer
\begin{equation}\label{bdry 0}
    \partial_t u_{i,\es} = \partial_t u_{i+1,\es}, \text{ on }(0,T)\times \Gamma_{i,i+1,\es},
\end{equation}
since
\begin{equation*}
    u_{i,\es}(t+h) - u_{i,\es}(t) = u_{i+1,\es}(t+h) - u_{i+1,\es}(t), 
\end{equation*}
on $\Gamma_{i,i+1,\es}$ for all $h>0$ such that $t+h \in (0,T)$.

Combing Equation~\eqref{cond: 2} and Equation~\eqref{cond: 1} we get
\begin{equation}\label{bdry 1}
     \mu_{i,\es} \nabla p_{i,\es} \cdot \boldsymbol{n} = \mu_{i+1,\es} \nabla p_{i+1,\es} \cdot \boldsymbol{n} \quad \text{ on } (0,T)\times \Gamma_{i,i+1,\es}.
\end{equation}
Moreover, Equation~\eqref{cond: 1} also implies 
\begin{align}
    \label{bdry 2} 
    \mu_{i,\es} p_{i,\es} \nabla u_{i,\es} \cdot \boldsymbol{n} = \mu_{i+1,\es} p_{i+1,\es} \nabla u_{i+1,\es} \cdot \boldsymbol{n} \quad \text{ on }(0,T)\times \Gamma_{i,i+1,\es},
\end{align}
which, combined with Equation~\eqref{cond: 2} gives also
\begin{align}
     \mu_{i,\es} \nabla u_{i,\es} \cdot \boldsymbol{n} = \mu_{i+1,\es} \nabla u_{i+1,\es} \cdot \boldsymbol{n} \quad \text{ on } (0,T)\times \Gamma_{i,i+1,\es}.
     \label{bdry 3}
\end{align}
Now we may come back to the computation of the term $\curlyJ$. By Equations~\eqref{bdry 0}, and \eqref{bdry 1} we directly infer that $\curlyJ_1$ vanishes.

We rewrite the term $\curlyJ_2$ as 
\begin{align*}
    &\sum_{i=1}^2 \int_{\Gamma_{i,i+1,\es}} \mu_{i,\es} p_{i,\es}  \ \sign(\partial_t u_{i,\es}) \ \partial_ t\prt*{\nabla u_{i,\es} \cdot \boldsymbol{n}} - \mu_{i+1,\es} p_{i+1,\es} \ \sign(\partial_t u_{i+1,\es}) \ \partial_t \prt*{\nabla u_{i+1,\es} \cdot \boldsymbol{n} }\\
    = & \underbrace{\sum_{i=1}^2 \int_{\Gamma_{i,i+1,\es}} \sign(\partial_t u_{i,\es}) \ \partial_t \prt*{\mu_{i,\es} p_{i,\es} \nabla u_{i,\es} \cdot \boldsymbol{n}  - \mu_{i+1,\es} p_{i+1,\es} \nabla u_{i+1,\es} \cdot \boldsymbol{n} }}_{\curlyJ_{2,1}}\\
    & - \underbrace{\sum_{i=1}^2 \int_{\Gamma_{i,i+1,\es}} \abs{\partial_t p_{i, \es}} \prt*{\mu_{i,\es} \nabla u_{i,\es} \cdot \boldsymbol{n} - \mu_{i+1,\es} \nabla u_{i+1,\es} \cdot \boldsymbol{n}}}_{\curlyJ_{2,2}} ,
\end{align*}
where we used Equation~\eqref{bdry 0}, which also implies $\partial_t p_{i,\es} = \partial_t p_{i+1,\es} $ on $(0,T) \times \Gamma_{i,i+1,\es}$, for  $i = 1,2$.
The terms $\curlyJ_{2,1}$ and  $\curlyJ_{2,2}$ vanish thanks to Equation~\eqref{bdry 2} and Equation~\eqref{bdry 3}, respectively.

Hence, from Equation~\eqref{eq: dt dtu}, we finally have
\begin{equation*}
      \ddt \sum_{i=1}^3  \int_{\Omega_{i,\es}} |\partial_t u_{i,\es}| \le G_M \sum_{i=1}^3  \int_{\Omega_{i,\es}} |\partial_t u_{i,\es}|,
\end{equation*}
and, using Gronwall's inequality, we obtain
\begin{equation*}
     \sum_{i=1}^3  \int_{\Omega_{i,\es}} |\partial_t u_{i,\es}(t)| \le e^{G_M t}\sum_{i=1}^3   \int_{\Omega_{i,\es}} |\prt*{\partial_t u_{i,\es}}^0|.
\end{equation*}
Thanks to the assumptions on the initial data, \cf Equation~\eqref{eq: assumptions initial data 2}, we conclude.

\CORR{\noindent{\bf(iii)}} $\boldsymbol{ p_\es \in L^2(0,T; H^1(\Omega\setminus \Omega_{2,\es})).}$ As known, in the context of a filtration equation, we can recover the pressure equation upon multiplying the equation on $u_{i,\es}$, \cf System~\eqref{epspb}, by $p'(u_{i,\es})=\gamma u_{i,\es}^{\gamma-1}$. Therefore, we obtain
\begin{equation}\label{eqp}
	\partial_t p_{i,\es} - \gamma \mu_{i,\es} p_{i,\es} \Delta p_{i,\es} = \mu_{i,\es} |\nabla p_{i,\es}|^2 + \gamma p_{i,\es} G(p_{i,\es}).
\end{equation}
Studying the equation on $p_\es$ rather than the equation on $u_\es$ turns out to be very useful in order to prove compactness, since, as it is well-know for the porous medium equation (PME), the gradient of the pressure can be easily bounded in $L^2$, while the density solution of the PME can develop jump singularities on the free boundary, \cite{vasquez}.

We integrate Equation~\eqref{eqp} on each $\Omega_{i,\es}$, and we sum over all $i$ to obtain
\begin{equation}\label{sump}
	\sum_{i=1}^{3} \int_{\Omega_{i,\es}}\partial_t p_{i,\es} = \sum_{i=1}^{3} \left(\gamma \mu_{i,\es} \int_{\Omega_i,\es} p_{i,\es}\Delta p_{i,\es} +\int_{\Omega_{i,\es}} \mu_{i,\es}|\nabla p_{i,\es}|^2 +\gamma\int_{\Omega_{i,\es}} p_{i,\es} G(p_{i,\es})\right).
\end{equation}
Integration by parts yields 
\begin{align*}
	\sum_{i=1}^{3} \mu_{i,\es} \int_{\Omega_{i,\es}} p_{i,\es}\Delta p_{i,\es} = &-\sum_{i=1}^{3} \mu_{i,\es} \int_{\Omega_{i,\es}} |\nabla p_{i,\es}|^2 + \sum_{i=1}^2 \int_{\Gamma_{i,i+1,\es}} \mu_{i,\es} p_{i,\es} \nabla p_{i,\es} \cdot \boldsymbol{n}_{i,i+1} \\
	&\qquad+\sum_{i=1}^2 \int_{\Gamma_{i,i+1,\es}} \mu_{i+1,\es} p_{i+1,\es} \nabla p_{i+1,\es} \cdot \boldsymbol{n}_{i+1,i} \\
	= &-\sum_{i=1}^{3}\mu_{i,\es} \int_{\Omega_{i,\es}} |\nabla p_{i,\es}|^2,
\end{align*}
since we have homogeneous Dirichlet boundary conditions on $\partial\Omega$ and the flux continuity conditions~\eqref{bdry 1}.
 
Hence, from Equation~\eqref{sump}, we have
\begin{equation}\label{sump2}
	\sum_{i=1}^{3} \int_{\Omega_{i,\es}}\partial_t p_{i,\es} = \sum_{i=1}^{3} \mu_{i,\es}\left((1-\gamma)\int_{\Omega_{i,\es}} |\nabla p_{i,\es}|^2 +\gamma\int_{\Omega_{i,\es}} p_{i,\es} G(p_{i,\es})\right).
\end{equation}
We integrate over time and we deduce that
\begin{equation}\label{sump2t}
	\sum_{i=1}^{3} \left( \int_{\Omega_{i,\es}} p_{i,\es}(T) - \int_{\Omega_{i,\es} }
	p_{i,\es}^0 + \mu_{i,\es} (\gamma-1) \int_{0}^{T}\int_{\Omega_{i,\es}} |\nabla p_{i,\es}|^2\right) = \sum_{i=1}^{3} \gamma \int_{0}^{T}\int_{\Omega_{i,\es}} p_{i,\es} G(p_{i,\es}).
\end{equation}
Finally, we conclude that
\begin{equation}\label{gradp}
	\sum_{i=1}^{3} \int_{0}^{T}\int_{\Omega_{i,\es}}\mu_{i,\es} |\nabla p_{i,\es}|^2 \le \sum_{i=1}^{3} \frac{\gamma}{\gamma-1} \int_{0}^{T}\int_{\Omega_{i,\es}} p_{i,\es} G(p_{i,\es}) + \frac{1}{\gamma-1} \int_{\Omega_{i,\es}} p_{i,\es}^0,
\end{equation}
Since we have already proved that $p_{i,\es}$ is bounded in $L^\infty(\Omega_T)$ and by assumption $G$ is continuous, we finally find that
\begin{equation}\label{eq:bound-nablap}
	\,\sum_{i=1}^{3} \mu_{i,\es} \int_{0}^{T}\int_{\Omega_{i,\es}} |\nabla p_{i,\es}|^2 \le C,	
\end{equation}
where $C$ denotes a constant independent of $\es$.
Since both $\mu_{1,\es}$ and $\mu_{3,\es}$ are bounded from below away from zero, we conclude that the uniform bound holds in $\Omega\setminus\Omega_{2,\es}$.

\end{proof}

\begin{remark}
Let us also notice that, differently from \cite{sanchez-palencia}, where the author studies the linear and uniformly parabolic case, proving weak compactness is not enough. Indeed, due to the presence of the nonlinear term $u \nabla p$, it is necessary to infer strong compactness of $u$. For this reason, the $L^1$-uniform estimate on the time derivative proven in Lemma~\ref{lemma: a priori} is fundamental.
\end{remark}

\section{Limit \texorpdfstring{$\es\rightarrow 0$}{}}\label{sec: limit}
We have now the {\it a priori} tools to face the limit $\es \rightarrow 0$. We need to construct an extension operator with the aim of controlling uniformly, with respect to $\es$, the pressure gradient in $L^2(\Omega)$. Indeed, from \eqref{eq:bound-nablap}, we see that one cannot find a uniform bound for $\norm{\nabla p_{2,\es} }_{L^2(\Omega_{2,\es})}$. The blow-up of Estimate~\eqref{eq:bound-nablap} for $i=2$, is in fact the main challenge in order to find compactness on $\Omega$. To this end, following \cite{sanchez-palencia}, we introduce in Subsection~\ref{subsec: extension} an extension operator which projects the points of $\Omega_{2,\es}$ inside $\Omega_{1,\es} \cup \Omega_{3,\es}$. Then, introducing proper test functions such that the variational formulation for $\es>0$ in \eqref{eq:variational-pb-homo} and $\es \rightarrow 0$ in \eqref{eq: variational limit problem} are well-defined, we can pass to the limit (Subsection~\ref{subsec: testfnc}).

\subsection{Extension operator and compactness}
\label{subsec: extension}
\begin{figure}[H]
    \centering
    \includegraphics[width=9cm]{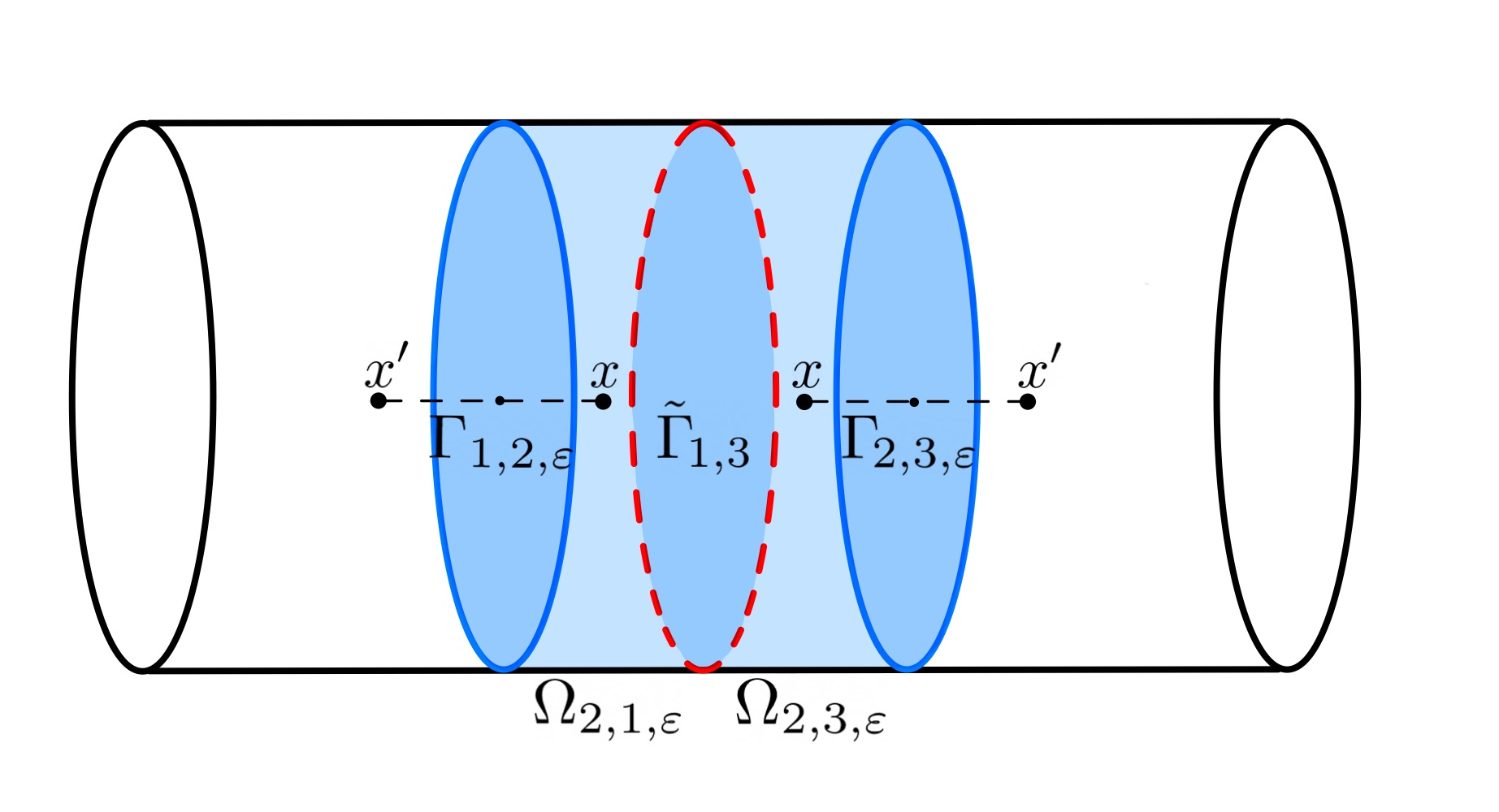} 
    \caption{Representation of the spatial symmetry used in the definition of the extension operator, \cf Equation~\eqref{eq: extension operator} and of the two subdomains of $\Omega_{2,1,\es}$ and $\Omega_{2,3,\es}$.}
    \label{fig: extension}
\end{figure}

As mentioned above, in order to be able to pass to the limit $\es \to 0$, we first need to define the following extension operator 
\begin{equation*}
    \curlyP_\es: L^q(0,T; W^{1,p}(\Omega\setminus \Omega_{2,\es})) \to L^q(0,T; W^{1,p}(\Omega \setminus \tilde\Gamma_{1,3} )), \quad \text{for} \quad 1\le p,q\le +\infty,
\end{equation*}
as follows for a general function $z \in  L^q(0,T; W^{1,p}(\Omega\setminus \Omega_{2,\es})) $, 
\begin{equation}\label{eq: extension operator}
    \curlyP_\es(z(t,x)) = \begin{cases}
        z(t,x),\quad &\text{if} \quad x\in {\Omega_{1,\es} \cup \Omega_{3,\es}},\\
       z(t,x'),\quad &\text{if} \quad x\in {\Omega_{2,\es}},
    \end{cases}
\end{equation}
where $x'$ is the symmetric of $x$ with respect to $\Gamma_{1,2,\es}$ (or $\Gamma_{2,3,\es}$) if $x \in \Omega_{2,1,\es}$ (respectively $x\in \Omega_{2,3,\es}$), \CORR{defined by the function $g:x \to x'$ for $x = (x_1,x_2,x_3)\in \Omega_{2,\es}$ such that 
\commentout{\[
    g(x) = \begin{cases}
        \left(-2d(\Gamma_{1,3,\epsilon}, x), x_2,x_3\right), \quad &\text{if} \quad x\in \Omega_{2,1,\epsilon},\\
        \left(2d(\Gamma_{2,3,\epsilon}, x), x_2,x_3\right), \quad  &\text{if} \quad x\in \Omega_{2,3,\epsilon},
    \end{cases}
\]}
\[
   g(x) = \begin{cases}
        \left(x_1,x_2,x_3-2\, d(\Gamma_{1,3,\es}, x)\right), \quad &\text{if} \quad x\in \Omega_{2,1,\es},\\
        \left(x_1,x_2,x_3+2\, d(\Gamma_{2,3,\es}, x) \right), \quad  &\text{if} \quad x\in \Omega_{2,3,\es},
    \end{cases}
\]
where $d(\Gamma_{1,2,\es}, x)$ (respectively $d(\Gamma_{2,3,\es}, x)$) denotes the distance between $x$ and the surface $\Gamma_{1,2,\es}$ (respectively $\Gamma_{2,3,\es}$).
The point $x'$ is illustrated in Figure~\ref{fig: extension}.} It can be easily seen that the function $g$ and its inverse have uniformly bounded first derivatives. Hence, we infer that $\curlyP_\es$ is linear and bounded, \ie
\[
    \norm{\mathcal{P}_\es(z)}_{L^q(0,T;W^{1,p}(\Omega \setminus \tilde\Gamma_{1,3}))} \le C,\quad  \forall z\in L^q(0,T; W^{1,p}(\Omega\setminus \Omega_{2,\es}))\CORR{, \mbox{ for } 1\le p,q\le \infty}.
\]
\commentout{Since the previous bound is uniform in $\es$, we obtain that $\curlyP_\es$ remains bounded even in the limit $\es \to 0$.}  
Let us notice that the extension operator is well defined also from $L^1((0,T)\times (\Omega \setminus \Omega_{2,\es}))$ into $L^1((0,T)\times ( \Omega \setminus \tilde\Gamma_{1,3}))$. Hence, we can apply it also on $u_{\es}$ and $\partial_t p_{\es}$.

\begin{remark}\label{rmk: a priori P} Thanks to the properties of the extension operator, the estimates stated in Lemma~\ref{lemma: a priori} hold true also upon applying $\curlyP_\es(\cdot)$ on $p_{\es}, u_\es$, and $\partial_t p_\es$, namely
\begin{align*} 
    &0\le \curlyP_\es(p_{\es})\le p_H, \quad    0\le\curlyP_\es(u_\es) \leq u_H,  \\[0.3em] 
   & \partial_t \curlyP_\es(p_\es) \in L^\infty(0,T; L^1(\Omega\setminus\tilde\Gamma_{1,3})),\\[0.3em]
   &\nabla\curlyP_\es(p_\es) \in L^2(0,T; L^2(\Omega\setminus\tilde\Gamma_{1,3})),\\[0.3em]
   &{\color{black}\frac{\gamma}{\gamma+1}  \nabla \prt*{ \curlyP_\es(u_\es^{\gamma+1})}\in L^2(0,T;L^2(\Omega\setminus\tilde\Gamma_{1,3}))},\\[0.3em
  ]
 & {\color{black}\partial_t(\curlyP_\es(u_\es^{\gamma+1}))\in L^\infty(0,T; L^1(\Omega\setminus \tilde\Gamma_{1,3})).} 
\end{align*}

{\noindent\color{black}The last two bounds hold thanks to the following arguments 
    \begin{equation*}
      \frac{\gamma}{\gamma+1}  \nabla \prt*{ \curlyP_\es(u_\es^{\gamma+1})} = \curlyP_\es(u_\es) \nabla\curlyP_\es(p_{\es})\in L^2(0,T;L^2(\Omega\setminus\tilde\Gamma_{1,3})),
    \end{equation*}
and
\begin{equation*}
   \partial_t\prt*{ \curlyP_\es(u_\es^{\gamma+1})}=(\gamma+1)\curlyP_\es(p_\es)\partial_t \curlyP_\es( u_\es)=(\gamma+1)\curlyP_\es(p_\es)\curlyP_\es(\partial_t  u_\es) \in L^\infty(0,T; L^1(\Omega\setminus \tilde\Gamma_{1,3}).
\end{equation*}}
\end{remark}

\begin{lemma}[Compactness of the extension operator]\label{lem: convergence P}
\begin{sloppypar}
 Let $(u_\es, p_\es)$ be the solution of Problem~\eqref{epspb}. There exists a couple $(\tilde{u},\tilde{p})$ with 
 \[
 \tilde{u}\in L^\infty(0,T; L^\infty(\Omega \setminus \tilde \Gamma_{1,3})), \quad{\tilde{p}\in L^2(0,T; H^1(\Omega \setminus \tilde \Gamma_{1,3})) \cap L^\infty(0,T; L^\infty(\Omega \setminus \tilde \Gamma_{1,3}))},
 \]
 such that, up to a subsequence, it holds  
\end{sloppypar}
 \begin{itemize}
   \item[\textit{(i)}] $\curlyP_\es(p_\es)\rightarrow  \tilde{p}$ strongly in $L^p(0,T; L^p(\Omega \setminus \tilde \Gamma_{1,3}))$, for $1\le p < +\infty$,
    \item[\textit{(ii)}] $\curlyP_\es(u_\es)\rightarrow  \tilde{u}$ strongly in $L^p(0,T; L^p(\Omega \setminus \tilde \Gamma_{1,3}))$, for $1\le p < +\infty$,
     \item[\textit{(iii)}] $\nabla\curlyP_\es(p_\es)\rightharpoonup  \nabla\tilde{p}$ weakly in $L^2(0,T; L^2(\Omega\setminus\tilde \Gamma_{1,3}))$.
 \end{itemize}
\end{lemma}
\begin{proof}
\textit{(i).} Since both $\partial_t \curlyP_\es(p_\es)$ and $\nabla \curlyP_\es(p_\es)$ are bounded in $L^1(0,T; L^1(\Omega \setminus \tilde \Gamma_{1,3}))$ uniformly with respect to $\es$, we infer the strong compactness of $\curlyP_\es(p_\es)$ in  $L^1(0,T; L^1(\Omega \setminus \tilde \Gamma_{1,3}))$. Let us also notice that since both $u_\es$ and $p_\es$ are uniformly bounded in $L^\infty(0,T; L^\infty(\Omega \setminus  \tilde \Gamma_{1,3}))$ then the strong convergence holds in any $L^p(0,T; L^p(\Omega \setminus \tilde \Gamma_{1,3}))$ with $1\le p < \infty.$

\textit{(ii).} From \textit{(i)}, we can extract a subsequence of $\curlyP_\es(p_\es)$ which converges almost everywhere. Then, remembering that $u_\es=p_\es^{1/\gamma}$, with $\gamma>1$ fixed, we have convergence of $\curlyP_\es(u_\es)$ almost everywhere. Thanks to the uniform $L^\infty$-bound of $\curlyP_\es(u_\es)$, Lebesgue's theorem implies the statement. Let us point out that, in particular, the $L^\infty$-uniform bound is also valid in the limit.

\textit{(iii).} The uniform boundedness of $ \nabla\curlyP_\es(p_\es)$ in $L^2(0,T; L^2(\Omega \setminus \tilde \Gamma_{1,3}))$ immediately implies weak convergence up to a subsequence. 

\end{proof}

\subsection{Test function space and passage to the limit \texorpdfstring{$\es\rightarrow 0$}{}}\label{subsec: testfnc}
Since in the limit we expect a discontinuity of the density on $\tilde\Gamma_{1,3}$, we need to define a suitable space of test functions. Therefore we construct the space $E^\star$ as follows.
Let us consider a function $\zeta \in \mathcal{D}(\Omega)$ (\ie $C^\infty_c(\Omega)$). 
For any $\es>0$ small enough, we build the function $v_\es = \curlyP_\es(\zeta)$, using the extension operator previously defined. The space of all linear combinations of these functions $v_\es$ is called $E^\star\subset H^1(\Omega \setminus \tilde{\Gamma}_{1,3})$\CORR{, namely
\[
E^\star=\left\{\sum_{n=1}^\infty c_n v_{\es,n}\; \mbox{ s.t. }\; c_n\in \R,\; v_{\es,n}=\curlyP_\es(\zeta_n),\; \zeta_n \in C^\infty_c(\Omega)\right\}.
\]
} 
We stress that the functions of $E^\star$ are discontinuous on $\tilde \Gamma_{1,3}$.

In the weak formulation of the limit problem~\eqref{eq: variational limit problem}, we will make use of piece-wise $C^\infty$-test functions (discontinuous on $\tilde\Gamma_{1,3}$) of the type $w(t,x)=\varphi(t)v(x)$, where $\varphi\in C^1([0,T))$ with $\varphi(T) = 0$ and $v\in E^*$. Therefore, $w$ belongs to $C^1([0,T);E^*)$.
On the other hand, in the variational formulation \eqref{eq:variational-pb-homo}, \ie for $\es>0$, $H^1(0,T; H^1_0(\Omega))$ test functions are required. Thus, in order to study the limit $\es\to 0$, we need to introduce a proper sequence of test functions depending on $\es$ that converges to $w$. 
To this end, we define the operator $L_\es: C^1([0,T); E^*) \to H^1(0,T; H_0^1(\Omega))$ such that
\[
    L_\es(w) \to w, \quad \text{uniformly as} \quad \es \to 0,  \quad \forall w \in C^1([0,T); E^\star).
\]
In this way, $L_\es(w)$ belongs to $H^1(0,T; H^1_0(\Omega))$, therefore, it can be used as test function in the formulation~\eqref{eq:variational-pb-homo}.

Following Sanchez-Palancia, \cite{sanchez-palencia}, for all $t\in[0,T]$ and $x = (x_1,x_2,x_3) \in \Omega$,  we define
\[
    L_\es(w(t,x)) = \begin{dcases}
        w(t,x),\qquad &\text{if}\quad x \notin \Omega_{2,\es},\\[0.3em]
        \frac12\left[w\prt*{t,x_1,x_2,\frac \es 2 } + w\prt*{t,x_1,x_2, - \frac \es2} \right] \\[0.3em]
        \quad +\left[w\prt*{t,x_1,x_2,\frac \es 2} - w\prt*{t,x_1,x_2, - \frac\es 2} \right]\frac{x_3}{\es}, \qquad &\text{otherwise}.
     \end{dcases}
\]
It can be easily verified that \CORR{$L_\es(w)$} is linear with respect to $x_3$ in $\Omega_{2,\es}$ and is continuous on $\partial \Omega_{2,\es}$. Let us notice that it holds
\begin{equation}\label{eq: dL dx3}
    \left|\frac{\partial L_\es(w)}{\partial x_3} \right|\le \frac C\es.
\end{equation} 
Furthermore, thanks to the mean value theorem, the partial derivatives of \CORR{$L_\es(w)$} with respect to $x_1$ and $x_2$ are bounded by a constant (independent of $\es$), 
\begin{equation*}
 \left|\frac{\partial L_\es(w)}{\partial x_1}\right|\le C,\qquad \left|\frac{\partial L_\es(w)}{\partial x_2}\right|\le C,    
\end{equation*}
and since the measure of $\Omega_{2,\es}$ is proportional to $\es$, we have 
\begin{equation}\label{eq: bound dxL}
     \int_0^T \int_{\Omega_{2,\es}} \left|\frac{\partial L_\es(w)}{\partial x_1}\right|^2 + \left|\frac{\partial L_\es(w)}{\partial x_2}\right|^2 \le C \es.
\end{equation}
Given $w\in C^1([0,T);E^\star)$, we take $L_\es(w)$ as a test function in the variational formulation of the problem, \ie Equation~\eqref{eq:variational-pb-homo}, and we have 
\begin{equation}
\label{eq: test L}
\begin{split}
        -\int_0^T\int_{\Omega} u_\es \partial_t L_\es(w) + \sum_{i=1}^3 \mu_{i,\es}  \int_0^T \int_{\Omega_{i,\es}} &u_{i,\es} \nabla p_{i,\es} \cdot \nabla L_{\es}(w) \\
        &=  \int_0^T\int_{\Omega} u_{\es}G(p_{\es}) L_\es(w)  + \int_{\Omega} u_\es^0 L_\es(w^0).
        \end{split}
\end{equation} 
Thanks to the \textit{a priori} estimates already proven, \cf Lemma~\ref{lemma: a priori}, Remark~\ref{rmk: a priori P} and the convergence result on the extension operator, \cf Lemma~\ref{lem: convergence P}, we are now able to pass to the limit $\es\rightarrow 0$ and recover the effective interface problem.
\begin{thm}\label{thm: existence}
For all test functions of the form $w(t,x):=\varphi(t)v(x)$ with $\varphi \in C^1([0,T))$ and $v\in E^*$, the limit couple $(\tilde{u},\tilde{p})$ of Lemma~\ref{lem: convergence P} satisfies the following equation
\begin{equation*}
\begin{split}
        -\int_0^T\int_{\Omega} \tilde{u} & \partial_t w + \tilde\mu_1\int_0^T\int_{\tilde\Omega_1}  \tilde{u}  \nabla \tilde{p} \cdot \nabla w +  \tilde\mu_3 \int_0^T\int_{\tilde\Omega_3}   \tilde{u}  \nabla \tilde{p} \cdot \nabla w \\
        &+ \tilde\mu_{1,3}\int_0^T\int_{\tilde\Gamma_{1,3}} \llbracket \Pi \rrbracket \prt*{w_{|x_3=0^+}- w_{|x_3=0^-}} = \int_0^T\int_{\Omega} \tilde{u}G(\tilde{p}) w + \int_{\Omega} \tilde{u}^0 w^0,
        \end{split}
\end{equation*}
where $$\llbracket \Pi \rrbracket:=\frac{\gamma}{\gamma+1}(\tilde{u}^{\gamma+1})_{|x_3=0^+}- \frac{\gamma}{\gamma+1}(\tilde{u}^{\gamma+1})_{|x_3=0^-},$$
\CORR{and $(\cdot)_{|x_3=0^-}= \curlyT_1 (\cdot)$ as well as $(\cdot)_{|x_3=0^+}= \curlyT_3 (\cdot)$, with $\curlyT_1, \curlyT_3$ the trace operators defined in Section~\ref{hp}}. \CORRdeux{By definition, this equation is the weak formulation of Problem~\eqref{effectivepb}. }
\end{thm}
\begin{proof}
We may pass to the limit in Equation~\eqref{eq: test L}, computing each term individually. 

\noindent{\textit{Step 1. Time derivative integral.}} We split the first integral into two parts
\begin{align*}
    -\int_0^T\int_{\Omega} u_\es \partial_t L_\es(w)  =   \underbrace{-\int_0^T\int_{\Omega_{1,\es}\cup \Omega_{3,\es}} u_\es \partial_t L_\es(w)}_{\curlyI_1} - \underbrace{ \int_0^T\int_{\Omega_{2,\es}} u_\es \partial_t L_\es(w)}_{\curlyI_2}.
\end{align*}
Since outside of $\Omega_{2,\es}$ the extension operator coincides with the identity, and $L_\es(w)=w$, we have
\begin{align*}
 \curlyI_1 = -\int_0^T\int_{\Omega_{1,\es}\cup \Omega_{3,\es}} \curlyP_\es(u_\es) \partial_t w = -\int_0^T\int_{\Omega} \curlyP_\es(u_\es) \partial_t w +\int_0^T\int_{\Omega_{2,\es}} \curlyP_\es(u_\es) \partial_t w.
\end{align*}
Thanks to Remark~\ref{rmk: a priori P}, we know that the last integral converges to zero, since both $\curlyP_\es(u_\es)$ and $\partial_t w$ are bounded in $L^2$ and the measure of $\Omega_{2,\es}$ tends to zero as $\es\rightarrow 0$.
Then, by Lemma~\ref{lem: convergence P}, we have
\begin{equation*}
     -\int_0^T\int_{\Omega} \curlyP_\es(u_\es) \partial_t w \longrightarrow  -\int_0^T\int_{\Omega} \tilde{u} \ \partial_t w, \quad \mbox{ as } \es\rightarrow 0, 
\end{equation*}
where we used the weak convergence of $\curlyP_\es(u_{\es})$ to $\tilde{u}$ in $L^2(0,T; L^2( \Omega\setminus\tilde\Gamma_{1,3}))$.
The term $\curlyI_2$ vanishes in the limit, since both $u_\es$ and $\partial_t L_\es(w)$ are bounded in $L^2$ uniformly with respect to $\es$. Hence, we finally have
 \begin{equation}\label{eq: conv of dt term}
      -\int_0^T\int_{\Omega} u_\es \partial_t L_\es(w)  \longrightarrow   -\int_0^T\int_{\Omega} \tilde{u} \ \partial_t w, \quad \mbox{ as } \es\rightarrow 0.
 \end{equation}
\noindent{\textit{Step 2. Reaction integral.}} We use the same argument for the reaction term, namely 
\begin{equation*}
    \int_0^T \int_{\Omega} u_\es G(p_\es) L_\es(w) =   \underbrace{\int_0^T \int_{\Omega_{1,\es}\cup\Omega_{3,\es}} u_\es G(p_\es) L_\es(w)}_{\curlyK_1} +  \underbrace{\int_0^T \int_{\Omega_{2,\es}} u_\es G(p_\es) L_\es(w)}_{\curlyK_2}.
\end{equation*}
Using again the convergence result on the extension operator, \cf Lemma~\ref{lem: convergence P}, we obtain
\begin{equation*}
    \curlyK_1 = \int_0^T \int_{\Omega_{1,\es}\cup\Omega_{3,\es}} \curlyP_\es(u_\es) G(\curlyP_\es(p_\es)) w \longrightarrow \int_0^T \int_{\Omega} \tilde{u}\, G(\tilde{p}) w, \quad \mbox{ as } \es\rightarrow 0,
\end{equation*}
since both $\curlyP_\es(u_\es)$ and $G(\curlyP_\es(p_\es))$ converge strongly in $L^2(0,T; L^2(\Omega\setminus\tilde\Gamma_{1,3}))$.
Arguing as before, it is immediate to see that $\curlyK_2$ vanishes in the limit. Hence
\begin{equation}
    \label{eq: conv of G term}
      \int_0^T \int_{\Omega} u_\es G(p_\es) L_\es(w) \longrightarrow   \int_0^T \int_{\Omega} \tilde{u} G(\tilde{p}) w, \quad \mbox{ as } \es\rightarrow 0.
\end{equation}
\noindent{\textit{Step 3. Initial data integral.}}
From \eqref{conv u0}, it is easy to see that
\begin{equation}\label{eq: conv term initial}
    \int_{\Omega} u_\es^0 L_\es(w^0) \longrightarrow     \int_{\Omega} \tilde{u}^0 w^0, \quad \mbox{ as } \es\rightarrow 0.
\end{equation}

\noindent{\textit{Step 4. Divergence integral.}}
Now it remains to treat the divergence term in Equation~\eqref{eq: test L}, from which we recover the effective interface conditions at the limit.

Since the extension operator $\curlyP_\es$ is in fact the identity operator on $\Omega\setminus\Omega_{2,\es}$, we can write
\begin{equation}\label{eq: H1 H2}
\begin{split}
  \sum_{i=1}^3 \mu_{i,\es}  &\int_0^T \int_{\Omega_{i,\es}} u_{i,\es} \nabla p_{i,\es} \cdot \nabla L_{\es}(w) \\
  =&\underbrace{\sum_{i=1,3} \mu_{i,\es} \int_0^T \int_{\Omega_{i,\es}} \curlyP_\es (u_{i,\es}) \nabla \curlyP_\es(p_{i,\es}) \cdot \nabla w}_{\curlyH_1} + \underbrace{\mu_{2,\es}\int_0^T \int_{\Omega_{2,\es}}  u_{2,\es} \nabla p_{2,\es} \cdot \nabla L_{\es}(w)}_{\curlyH_2}. 
  \end{split}
\end{equation}
We treat the two terms separately. 
Since we want to use the weak convergence of $\nabla\curlyP_\es(p_\es)$ in $L^2(0,T;L^2(\Omega\setminus \tilde\Gamma_{1,3}))$ (together with the strong convergence of $\curlyP_\es(u_\es)$ in $L^2(0,T;L^2(\Omega\setminus \tilde\Gamma_{1,3}))$) we need to write the term $\curlyH_1$ as an integral over $\Omega$. To this end, let $\overline{\mu}_\es:=\overline{\mu}_\es(x)$ be a function defined as follows
\begin{equation*}
    \overline{\mu}_\es (x) := 
    \begin{dcases}
    \mu_{1,\es} \qquad &\text{ for } x \in \Omega_{1,\es},\\[0.2em]
    0  \qquad &\text{ for } x \in \Omega_{2,\es},\\[0.2em]
    \mu_{3,\es} \qquad &\text{ for } x \in \Omega_{3,\es}.
    \end{dcases}
\end{equation*}  
Then, we can write
\begin{equation*}
     \curlyH_1 = \int_0^T \int_{\Omega} \overline{\mu}_\es \curlyP_\es (u_\es) \nabla \curlyP_\es(p_\es) \cdot \nabla w .
\end{equation*}
Let us notice that as $\es$ goes to $0$, $\overline{\mu}_\es$ converges to $\tilde\mu_1$ in $\tilde\Omega_1$ and $\tilde\mu_3$ in $\tilde\Omega_3$. 
Therefore, by Lemma~\ref{lem: convergence P}, we infer
\begin{equation}\label{eq: conv divergence term}
    \curlyH_1 \longrightarrow \tilde\mu_1 \CORR{\int_0^T\int_{\tilde\Omega_1}  \ \tilde{u} \ \nabla \tilde{p} \cdot \nabla w +  \tilde\mu_3  \int_0^T\int_{\tilde\Omega_3} \ \tilde{u} \ \nabla \tilde{p} \cdot \nabla w},\quad \text{as} \quad \es\to 0. 
\end{equation}
Now we treat the term $\curlyH_2$, which can be written as 
\begin{align*}
   \curlyH_2 =& \mu_{2,\es}\int_0^T \int_{\Omega_{2,\es}}  u_{2,\es} \nabla p_{2,\es} \cdot \nabla L_{\es}(w)\\
   =& \mu_{2,\es}\int_0^T\int_{\Omega_{2,\es}} \prt*{u_{2,\es} \frac{\partial p_{2,\es}}{\partial x_1} \frac{\partial L_\es(w)}{\partial x_1} + u_{2,\es} \frac{\partial p_{2,\es}}{\partial x_2} \frac{\partial L_\es(w)}{\partial x_2}} +  \mu_{2,\es}\int_0^T\int_{\Omega_{2,\es}} u_{2,\es} \frac{\partial p_{2,\es}}{\partial x_3} \frac{\partial L_\es(w)}{\partial x_3}.
\end{align*}
By the Cauchy-Schwarz inequality, the a priori estimate~\eqref{eq:bound-nablap}, and Equation~\eqref{eq: bound dxL}, we have
\begin{align*}
 \mu_{2,\es}\int_0^T\int_{\Omega_{2,\es}}& u_{2,\es} \frac{\partial p_{2,\es}}{\partial x_1} \frac{\partial L_\es(w)}{\partial x_1} + u_{2,\es} \frac{\partial p_{2,\es}}{\partial x_2} \frac{\partial L_\es(w)}{\partial x_2} \\[0.3em]
 &\leq \mu_{2,\es}^{1/2} \|u_{2,\es}\|_{L^\infty((0,T)\times\Omega_{2,\es})} \prt*{\left\|\mu_{2,\es}^{1/2} \frac{\partial p_{2,\es}}{\partial x_1}\right\|_{L^2((0,T)\times\Omega_{2,\es})}  \left\|\frac{\partial L_\es(w)}{\partial x_1} \right\|_{L^2((0,T)\times\Omega_{2,\es})} }\\[0.3em]
 &\quad +\mu_{2,\es}^{1/2} \|u_{2,\es}\|_{L^\infty((0,T)\times\Omega_{2,\es})} \prt*{\left\|\mu_{2,\es}^{1/2} \frac{\partial p_{2,\es}}{\partial x_2}\right\|_{L^2((0,T)\times\Omega_{2,\es})}  \left\|\frac{\partial L_\es(w)}{\partial x_2} \right\|_{L^2((0,T)\times\Omega_{2,\es})} }\\[0.3em]
 &\leq C\ \mu_{2,\es}^{1/2} \ \es^{1/2} \rightarrow 0.
\end{align*}
On the other hand, by Fubini's theorem, the following equality holds
\begin{equation*}
\begin{split}
     \mu_{2,\es}&\int_0^T\int_{\Omega_{2,\es}} u_{2,\es} \frac{\partial p_{2,\es}}{\partial x_3} \frac{\partial L_\es(w)}{\partial x_3} \\[0.3em]
     &=  \mu_{2,\es}\frac{\gamma}{\gamma+1}\int_0^T\int_{\Omega_{2,\es}}  \frac{\partial u_{2,\es}^{\gamma+1}}{\partial x_3} \frac{\partial L_\es(w)}{\partial x_3} \\[0.3em]
     &= \mu_{2,\es}\frac{\gamma}{\gamma+1}\int_0^T \int_{-\es/2}^{\es/2}\int_{\tilde\Gamma_{1,3}}  \frac{\partial u_{2,\es}^{\gamma+1}}{\partial x_3} \frac{\partial L_\es(w)}{\partial x_3} \dx{\sigma} \dx{x_3}\\[0.3em]
   &= \mu_{2,\es}\frac{\gamma}{\gamma+1}\int_0^T\int_{-\es/2}^{\es/2}\int_{\tilde\Gamma_{1,3}}  \frac{\partial u_{2,\es}^{\gamma+1}}{\partial x_3} \frac{w_{|x_3=\frac \es 2}- w_{|x_3=-\frac \es 2}}{\es} \dx{\sigma} \dx{x_3}\\[0.3em]
    &=\frac{\mu_{2,\es}}{\es}\frac{\gamma}{\gamma+1}\int_0^T \int_{\tilde\Gamma_{1,3}} \prt*{w_{|x_3=\frac \es 2}- w_{|x_3=-\frac \es 2}} \int_{-\es/2}^{\es/2}  \frac{\partial u_{2,\es}^{\gamma+1}}{\partial x_3} \dx{x_3} \dx{\sigma} \\[0.3em]
     &= \frac{\mu_{2,\es}}{\es}\frac{\gamma}{\gamma+1}\int_0^T\int_{\tilde\Gamma_{1,3}} \prt*{(u_{2,\es}^{\gamma+1})_{|x_3=\frac \es 2}- (u_{2,\es}^{\gamma+1})_{|x_3=-\frac \es 2}}\cdot \prt*{w_{|x_3=\frac \es 2}- w_{|x_3=-\frac \es 2}}.
     \end{split}
     \end{equation*}
Therefore,
\begin{equation}\label{eq: fund eq}
    \lim_{\es\to 0} \curlyH_2 = \lim_{\es\to 0} \frac{\mu_{2,\es}}{\es}\frac{\gamma}{\gamma+1}\int_0^T\int_{\tilde\Gamma_{1,3}} \prt*{(u_{2,\es}^{\gamma+1})_{|x_3=\frac \es 2}- (u_{2,\es}^{\gamma+1})_{|x_3=-\frac \es 2}}\cdot \prt*{w_{|x_3=\frac \es 2}- w_{|x_3=-\frac \es 2}}.
\end{equation}
In order to conclude the proof, we state the following lemma, which is proven \hyperlink{proof lemma: intermediate}{below}.

\begin{lemma}\label{lemma: intermediate}
The following limit holds uniformly in $\tilde\Gamma_{1,3}$
\begin{equation}\label{membrtest}
    w_{|x_3=\frac \es 2}- w_{|x_3=-\frac \es 2} \longrightarrow w_{|x_3=0^+}- w_{|x_3=0^-}, \quad\mbox{ as } \es \rightarrow 0.
\end{equation} 
Moreover, 
\begin{equation}\label{eq: conv u gamma+1}
  \frac{\gamma}{\gamma+1} \prt*{ (u_{2,\es}^{\gamma+1})_{|x_3=\frac \es 2}- (u_{2,\es}^{\gamma+1})_{|x_3=-\frac \es 2}} \longrightarrow \frac{\gamma}{\gamma+1}  \prt*{(\tilde{u}^{\gamma+1})_{|x_3=0^+}- (\tilde{u}^{\gamma+1})_{|x_3=0^-}},
\end{equation}
strongly in $L^2(0,T; L^2(\tilde\Gamma_{1,3}))$, as $\es \rightarrow 0$.
\end{lemma}
We may finally find the limit of the term $\curlyH_2$, using Assumption~\eqref{eq:cond-mu}, and applying Lemma~\ref{lemma: intermediate} to Equation~\eqref{eq: fund eq}
\begin{align*}
  &\frac{\mu_{2,\es}}{\es}\frac{\gamma}{\gamma+1}\int_0^T\int_{\tilde\Gamma_{1,3}} \prt*{(u_{2,\es}^{\gamma+1})_{|x_3=\frac \es 2}- (u_{2,\es}^{\gamma+1})_{|x_3=-\frac \es 2}}\cdot \prt*{w_{|x_3=\frac \es 2}- w_{|x_3=-\frac \es 2}}\\
  \longrightarrow \; & \tilde\mu_{1,3}\frac{\gamma}{\gamma+1}\int_0^T\int_{\tilde\Gamma_{1,3}} \prt*{ \left(\tilde{u}^{\gamma+1}\right)_{|x_3=0^+}- \left(\tilde{u}^{\gamma+1}\right)_{|x_3=0-}}\cdot \prt*{w_{|x_3=0^+}- w_{|x_3=0^-}},
\end{align*}
as $\es\to 0$.
Combining the above convergence to Equation~\eqref{eq: H1 H2} and Equation~\eqref{eq: conv divergence term}, we find the limit of the divergence term as $\es$ goes to $0$,
\begin{align*}
     \sum_{i=1}^3 \mu_{i,\es}  &\int_0^T \int_{\Omega_{i,\es}} u_{i,\es} \nabla p_{i,\es} \cdot \nabla L_{\es}(w) \\
     \longrightarrow \;& \tilde\mu_1  \int_0^T\int_{\tilde\Omega_1} \ \tilde{u} \nabla \tilde{p} \cdot \nabla w + \tilde\mu_3 \int_0^T\int_{\tilde\Omega_3}  \tilde{u} \nabla \tilde{p} \cdot \nabla w\\
     &+ \tilde\mu_{1,3}\frac{\gamma}{\gamma+1}\int_0^T\int_{\tilde\Gamma_{1,3}} \prt*{(\tilde{u}^{\gamma+1})_{|x_3=0^+}- (\tilde{u}^{\gamma+1})_{|x_3=0^-}}\cdot \prt*{w_{|x_3=0^+}- w_{|x_3=0^-}},
\end{align*}
which, together with Equations~\eqref{eq: test L}, \eqref{eq: conv of dt term}, \eqref{eq: conv of G term}, and \eqref{eq: conv term initial}, concludes the proof.

\end{proof}

We now turn to the \hypertarget{proof lemma: intermediate}{proof of Lemma~\ref{lemma: intermediate}}
\begin{proof}[Proof of Lemma~\ref{lemma: intermediate}]
\begin{sloppypar}
    Since by definition $w(t,x)= \varphi(t) v(x) $, with $\varphi\in C^1([0,T))$ and $v\in E^*$, the uniform convergence in Equation~\eqref{membrtest} comes from the piece-wise differentiability of $w$. 
    
    A little bit trickier is the second convergence, \ie Equation \eqref{eq: conv u gamma+1}. We recall that on ${ \{x_3= \pm \es /2 \}}$, ${u^{\gamma+1}_{2,\es}}$ coincides with ${\curlyP_\es(u_{\es}^{\gamma+1})}$, since 
    across the interfaces $u_\es$ is continuous and ${\curlyP_\es(u_{i,\es})=u_{i,\es}}$, for $i=1,3$. 
    \end{sloppypar}
    \color{black}{Let us recall that from Remark~\ref{rmk: a priori P}, we have
  $$  \left\|\curlyP_\es(u_\es^{\gamma+1})\right\|_{L^2(0,T; H^1(\Omega\setminus\tilde\Gamma_{1,3}))} \leq C, \ \text{ and } \
    \left\|\partial_t\prt*{ \curlyP_\es(u_\es^{\gamma+1})}\right\|_{L^\infty(0,T;L^1(\Omega\setminus\tilde\Gamma_{1,3}))}\leq C. $$  }
    
    \commentout{Moreover, from \textit{(i)} and \textit{(iv)} in Remark~\ref{rmk: a priori P}, we infer that
    \begin{equation*}
      \frac{\gamma}{\gamma+1}  \nabla \prt*{ \curlyP_\es(u_\es^{\gamma+1})} = \curlyP_\es(u_\es) \nabla\curlyP_\es(p_{\es})\in L^2(0,T;L^2(\Omega\setminus\tilde\Gamma_{1,3})),
    \end{equation*}
    hence 
    \begin{equation*}
         \CORR{\left\|\curlyP_\es(u_\es^{\gamma+1})\right\|_{L^2(0,T; H^1(\Omega\setminus\tilde\Gamma_{1,3}))} \leq C.}
    \end{equation*}
Again from Remark~\ref{rmk: a priori P}, we know that
\begin{equation*}
    \left\|\partial_t\prt*{ \curlyP_\es(u_\es^{\gamma+1})}\right\|_{L^\infty(0,T;L^1(\Omega\setminus\tilde\Gamma_{1,3}))}\leq C.
\end{equation*}}

\noindent
Since we have the following embeddings
\begin{equation*}
    H^1(\Omega\setminus\tilde\Gamma_{1,3}) \subset\subset H^\beta(\Omega\setminus\tilde\Gamma_{1,3})\subset L^1(\Omega\setminus\tilde\Gamma_{1,3}),
\end{equation*}
for every $\frac 12 < \beta < 1$, upon applying Aubin-Lions lemma, \cite{aubin, lions}, we obtain
\begin{equation*}
    \curlyP_\es(u_\es^{\gamma +1}) \longrightarrow \tilde{u}^{\gamma +1}, \quad \text{ as } \es \to 0,
\end{equation*}
strongly in $L^2(0,T; H^\beta (\Omega\setminus\tilde\Gamma_{1,3}))$.

\CORR{
Thanks to the continuity of the trace operators $\curlyT_\alpha:  H^\beta(\tilde \Omega_\alpha \setminus\tilde\Gamma_{1,3}) \rightarrow{} L^2(\partial \tilde \Omega_\alpha)$, 
for $\frac 12 < \beta < 1$ and $\alpha =1,3$, we finally recover that
\begin{equation}\label{eq: trace convergence}
     \left\|\curlyP_\es(u_\es^{\gamma+1})_{|x_3= 0^\pm}- \left(\tilde{u}^{\gamma+1}\right)_{|x_3=0^\pm}\right\|_{L^2(0,T; L^2(\tilde\Gamma_{1,3}))} \le C \left\| \curlyP_\es(u_\es^{\gamma+1})- \tilde{u}^{\gamma+1}\right\|_{L^2(0,T; H^\beta(\Omega\setminus\tilde\Gamma_{1,3}))}\rightarrow 0,
\end{equation}
as $\es \rightarrow 0$.
We recall that the trace vanishes on the external boundary, $\partial \Omega$, therefore  we only consider the $L^2(0,T; L^2(\tilde\Gamma_{1,3}))$-norm.
}

Recalling that $L$ is the length of $\Omega$, trivially, we find the following estimate
\begin{align*}
       \big\|\curlyP_\es(u_{\es}^{\gamma+1})_{|x_3=\pm \es/2}  - \curlyP_{\es}(u_{\es}^{\gamma+1})_{|x_3=0^\pm}& \big\|^2_{L^2(0,T; L^2(\tilde\Gamma_{1,3}))}\\
       &= \int_0^T\int_{\tilde\Gamma_{1,3}} \prt*{\int_{0}^{\pm\es/2} \frac{\partial \curlyP_{\es}(u_{\es}^{\gamma+1})}{\partial x_3}  }^2   \\[0.3em]
       & = \int_0^T \int_{\tilde \Gamma_{1,3}} \left( \int_L   \frac{\partial \curlyP_{\es}(u_{\es}^{\gamma+1})}{\partial x_3} \mathds{1}_{[0,\pm\es/2]}(x_3) \right)^2 \\[0.3em]
       &\le \int_0^T \int_{\tilde \Gamma_{1,3}} \left( \int_L   \left( \frac{\partial \curlyP_{\es}(u_{\es}^{\gamma+1})}{\partial x_3}\right)^2 \int_L \left(\mathds{1}_{[0,\pm\es/2]}(x_3) \right)^2 \right)\\[0.3em]
       &\leq  \frac\es 2  \|\nabla \curlyP_\es(u_{\es}^{\gamma+1})\|_{L^2(0,T;L^2(\Omega\setminus\tilde\Gamma_{1,3}))}\\[0.3em]
       &\leq \es \ C,
   \end{align*} 
   and combing it with Equation~\eqref{eq: trace convergence}, we finally obtain Equation~\eqref{eq: conv u gamma+1}.
  
\end{proof}

\begin{remark}
Although not relevant from a biological point of view, let us point out that, in the case of dimension greater than 3, the analysis goes through without major changes. It is clear that the \textit{a priori} estimates are not affected by the shape or the dimension of the domain (although some uniform constants $C$ may depend on the dimension, this does not change the result in Lemma~\ref{lemma: a priori}). The following methods, and in particular the definition of the extension operator and the functional space of test functions, clearly depends on the dimension, but the strategy is analogous for a $d$-dimensional cylinder with axis $\{x_1=\dots=x_{d-1}=0\}$.
\end{remark}

\CORR{\begin{remark}
We did not consider the case of non-constant mobilities, \ie $\mu_{i,\es}:=\mu_{i,\es}(x)$, but continuity and boundedness are the minimal hypothesis to succeed in the proof.
\end{remark}}

\section{Conclusions and perspectives}\label{sec: conclusion}
We proved the convergence of a continuous model of cell invasion through a membrane when its thickness is converging to zero, hence giving a rigorous derivation of the effective transmission conditions already conjectured in Chaplain \textit{et al.}, \cite{giverso}. Our strategy relies on the methods developed in \cite{sanchez-palencia}, although we had to handle the difficulties coming from the nonlinearity and degeneracy of the system. \commentout{We did not consider the case of non-constant mobilities, \ie $\mu_{i,\es}:=\mu_{i,\es}(t,x)$, which could bring additional challenges to the derivation of the \textit{a priori estimates} and the compactness results. In particular, a}\CORR{A} very interesting direction both from the biological and mathematical point of view, could be coupling the system to an equation describing the evolution of the MMP concentration. In fact, as observed in \cite{giverso}, the permeability coefficient can depend on the local concentration of MMPs, since it indicates the level of "aggressiveness" at which the tumour is able to destroy the membrane and invade the tissue.

In a recent work~\cite{giverso2}, a formal derivation of the multi-species effective problem has been proposed. However, its rigorous proof remains an interesting and challenging open question. Indeed, introducing multiple species of cells, hence dealing with a cross-(nonlinear)-diffusion system, adds several challenges to the problem. As it is well-known, proving the existence of solutions to cross-diffusion systems with different mobilities is one of the most challenging and still open questions in the field. Nevertheless, even when dealing with the same constant mobility coefficients, the nature of the multi-species system (at least for dimension greater than one) usually requires strong compactness on the pressure gradient. We refer the reader to \cite{GPS, price2020global} for existence results of the two-species model without membrane conditions.

Another direction of further investigation of the effective transmission problem \eqref{effectivepb} could be studying the so-called \textit{incompressible limit}, namely the limit of the system as $\gamma \rightarrow \infty$. The study of this limit has a long history of applications to tumour growth models, and has attracted a lot of interest since it links density-based models to a geometrical (or free boundary) representation, \cf \cite{PQV, KP17}.

Moreover, including the heterogeneity of the membrane in the model could not only be useful in order to improve the biological relevance of the model, but could bring interesting mathematical challenges, forcing to develop new methods or adapt already existent ones, \cite{neuss-radu}, from the parabolic to the degenerate case.

\section*{Acknowledgements}
G.C. and A.P. have received funding from the European Research Council (ERC) under the European Union's Horizon 2020 research and innovation programme (grant agreement No 740623). The work of G.C. was also partially supported by GNAMPA-INdAM.
\\
N.D. has received funding from the European Union's Horizon 2020 research and innovation program under the Marie Skłodowska-Curie (grant agreement No 754362). \\
The authors are grateful to Beno\^it Perthame for fruitful discussions.

\CORRdeux{
\appendix
\section{Existence of weak solution of the initial problem}\label{appendix:existence}
We prove in this appendix the existence of solution for System~\eqref{epspb}. 
Similarly to diffraction problems modelled by linear parabolic equations (see Section 3.13 in~\cite{ladyvzenskaja1988linear}), this result follows from the existence of solution for the Porous Medium Equation with discontinuous coefficients. Indeed, using a test function $w \in C^\infty(\Omega_T)$, solutions of the following weak formulation 
\[
    \int_\Omega \partial_t u  w + \mu(x) u \nabla u^\gamma \cdot \nabla w \,\dd x= \int_\Omega u G(p) w \, \dd x,  
\]
are actually solutions of the strong form~\eqref{epspb}. This is obtained from the fact that the interfaces $\Gamma_{i,i+1}$ (for $i=1,2$) are continuous and from the interface conditions.   

Even though the proof of the existence of weak solutions follow the lines of Section 5.4 in~\cite{vasquez}, we could not find a proof of this result in the case of discontinuous mobility coefficients in the literature, hence, for the sake of clarity, we give in this appendix the idea of the proof.  

\begin{thm}[Existence of weak solutions for the initial problem]
    Assuming that $\mu_i >0$ for $i=1,2,3$, System~\eqref{epspb} admits a weak solution $u \in L^1(\Omega_T)$ and $p \in L^1(0,T;H^1_0(\Omega))$.
\end{thm}
\begin{proof}
    \textit{Step 1: Regularized problem.} We first regularize the model to convert it into a non-degenerate parabolic model. We use a positive parameter $n$ and define a positive initial condition 
    \begin{equation}
    u_{0n} = u_{0} + \frac{1}{n}.
    \end{equation}
    Our regularized problem reads 
    \begin{equation}\label{epspb-reg}
	\left\{
	\begin{array}{rlll}
		&\partial_t u_{i,n} - \mu_{i} \nabla \cdot ( u_{i,n} \nabla p_{i,n}) = u_{i,n} G(p_{i,n})  & \text{ in } (0,T)\times\Omega_{i}, & i=1,2,3,\\[1em] 
		&\mu_{i} u_{i,n} \nabla p_{i,n}\cdot \boldsymbol{n}_{i,i+1} =  \mu_{i+1} u_{i+1,n} \nabla p_{i+1,n}\cdot \boldsymbol{n}_{i,i+1}  &\text{ on } (0,T)\times\Gamma_{i,i+1,}, & i = 1,2,\\[1em]
		&u_{i,n} = u_{i+1,n}  &\text{ on } (0,T) \times\Gamma_{i,i+1,}, & i = 1,2,\\[1em]
		&u_{i,n}=\frac{1}{n} &\text{ on } (0,T) \times\partial\Omega.
	\end{array}
	\right.
\end{equation}

From results on diffraction problems from~\cite{ladyvzenskaja1988linear} we know that in weak form our regularized problem is only a quasi-linear parabolic PDE. Thus, from standard results on these equations, we can have the existence of a classical solution $u_n \in  C^{1,2}(\Omega_T)$ of Problem~\eqref{epspb-reg}. Then, at this point the rest of the proof is similar to Section 5.4 in~\cite{vasquez}. 
We obtain at the end the existence of weak solutions $u \in L^1(\Omega_T)$ and $p \in L^1(0,T;H^1_0(\Omega))$ of Problem~\eqref{epspb}.

\end{proof}
}
\bibliographystyle{abbrvnat}
\bibliography{references}

\end{document}